\documentclass[11pt,reqno,a4paper]{amsart}
\usepackage{amsfonts}
\usepackage{ifthen}
\usepackage{amsthm,cite}
\usepackage{amsmath,mathrsfs}
\usepackage{graphicx}
\usepackage{subcaption}
\usepackage{amscd,amssymb,amsthm}
\usepackage{color}
\usepackage{hyperref}
\usepackage{amsthm}
\usepackage{amsmath}

\allowdisplaybreaks

\usepackage{array,multirow,makecell}

\setcellgapes{1pt}
\makegapedcells
\newcolumntype{R}[1]{>{\raggedleft\arraybackslash }b{#1}}
\newcolumntype{L}[1]{>{\raggedright\arraybackslash }b{#1}}
\newcolumntype{C}[1]{>{\centering\arraybackslash }b{#1}}
\newcounter{minutes}
\divide\time by 60
\newcounter{hours}
\multiply\time by 60 \addtocounter{minutes}{-\time}

\newtheorem{theo}{Theorem}[section]

\newtheorem{theorem}{Theorem}[section]

\newtheorem{definition}[theorem]{Definition}
\newtheorem{corollary}[theorem]{Corollary}
\newtheorem{remark}[theorem]{Remark}
\newtheorem{example}[theorem]{Example}

\usepackage{tikz,xcolor,hyperref}
\definecolor{lime}{HTML}{A6CE39}
\DeclareRobustCommand{\orcidicon}{%
	\begin{tikzpicture}
	\draw[lime, fill=lime] (0,0)
	circle [radius=0.16]
	node[white] {{\fontfamily{qag}\selectfont \tiny ID}};
	\draw[white, fill=white] (-0.0625,0.095)
	circle [radius=0.007];
	\end{tikzpicture}
	\hspace{-2mm}
}

\foreach \x in {A, ..., Z}{%
	\expandafter\xdef\csname orcid\x\endcsname{\noexpand\href{https://orcid.org/\csname orcidauthor\x\endcsname}{\noexpand\orcidicon}}
}

\title{ Bicomplex Generalized  Hypergeometric Functions and their Applications}
\author[Snehasis Bera]{Snehasis Bera}
\address{Snehasis Bera\newline Department of Mathematics,\newline National Institute of Technology Jamshedpur, Jamshedpur-831014, Jharkhand, India.}
\email{berasnehasis1996@gmail.com}
\author[Sourav Das]{Sourav Das$^\ast$\orcidA{}}
\thanks{$^\ast$Corresponding author}
\address{Sourav Das\newline Department of Mathematics,\newline National Institute of Technology Jamshedpur, Jamshedpur-831014, Jharkhand, India.}
\email{souravdasmath@gmail.com, souravdas.math@nitjsr.ac.in}
\author[Abhijit Banerjee]{Abhijit Banerjee}
\address{Abhijit Banerjee\newline Department of Mathematics,\newline Garhbeta College, Paschim Medinipur-721127, West Bengal, India.}
\email{abhijit.banerjee.81@gmail.com}
\date{}
\begin{document}

\maketitle

\begin{abstract}
  In this work, generalized hypergeometric functions for bicomplex argument is introduced and its convergence criteria is derived. Furthermore, integral representation of this function has been established. Moreover, quadratic transformation, differential relation, analyticity and contiguous relations of this function are derived. Additionally, applications in quantum information system and quantum optics are provided as a consequence.
\end{abstract}
\section{Introduction}
In 1892, C. Segre \cite{le} introduced the bicomplex number. G. B. Price \cite{multicomplex} discussed bicomplex numbers based on multi-complex spaces and functions. This bicomplex numbers are useful in fluid mechanics and mathematical physics.
The set of bicomplex numbers is defined as \cite{bc_number}: $$\mathbb{BC}=\{Z=z+i_2z'|z, z'\in \mathbb{C}(i)\},$$
where $\mathbb{C}(i)$ is the set of complex numbers with the imaginary unit $i_1$ and $i_1i_2=i_2i_1=k,i_1^2=i_2^2=-1,k^2=1$. If we set $z'=0$ then we get our ordinary complex numbers. Under the addition and multiplication of two bicomplex numbers $Z= z +i_2z'$ and $Z_1=z_1+i_2z'_1$ defining by
\begin{eqnarray*}
Z+Z_1&=&(z +z_1)+i_2(z' +z'_1)\\
Z\cdot Z_1&=&(z z_1 -z' z'_1)+i_2(z'z_1 +z z'_1)
\end{eqnarray*}
 $\mathbb{BC}$ forms a commutative ring with unity. The set of all zero divisors in $\mathbb{BC}$ is called null cone  defined by \cite{bc_number} : \begin{align*}
    \mathbb{O}_2=\{Z=z+i_2z'|{z}^2+{z'}^2=0\}. \end{align*}
 In particular the numbers
\begin{align}\label{eq:idm}
e_1=\frac{1+k}{2} \quad  \mbox{and} \quad e_2=\frac{1-k}{2}
\end{align}
are zero divisors, which are linearly independent in the  space $\mathbb{BC}$, and satisfies the identities :
$ e_1+e_2=1, \; e_1.e_2=0, \; e_1-e_2=k, \; e_1^{2}=e_1, \; e_2^{2}=e_2.$
For any bicomplex number $Z=z+i_2z' \in \mathbb{BC} $ can be represented by  $$ Z= z_1  e_1+z_2 e_2,$$ where  $z_1=z-i_1z'$ and $z_2=z+i_1z'$. This representation is called idempotent representation of the bicomplex number $Z$ and the set $\{e_1,e_2\}$ is known as idempotent basis of the bicomplex module $\mathbb{BC}$.

The set of all hyperbolic numbers is defined as follows \cite{bc_holomorphic} :
  $$\mathbb{D}=\{Z=x+ky|x,y\in{\mathbb{R}}\}.$$ The following two subsets of $\mathbb{D}$ $$\mathbb{D}^+=\{x+ky|x^2-y^2\geq0,x\geq0\} \quad \mbox{and} \quad \mathbb{D}^-=\{x+ky|x^2-y^2\geq0,x\leq0\}$$ are called non-negative hyperbolic numbers and non-positive hyperbolic numbers respectively.
  The projection mappings $P_1:\mathbb{BC}\rightarrow A\subseteq\mathbb{C}$ and $P_2:\mathbb{BC}\rightarrow B\subseteq\mathbb{C}$ are defined by
\begin{align*}
P_1(Z)=P_1(z_1e_1+z_2e_2)=z_1=z-i_1z'\in A\\
P_2(Z)=P_2(z_1 e_1+z_2 e_2)=z_2=z+i_1z'\in B,
\end{align*}
where $A=\{z_1=z-i_1z'|z,z'\in\mathbb{C}\}$ and $B=\{z_2=z+i_1z'|z,z'\in\mathbb{C}\}$ are two $\mathbb{C}(i_1)$ linear spaces. Then every bicomplex number $Z$ can be uniquely represented as $$Z=P_1(Z)e_1+P_2(Z)e_2.$$

Every bicomplex numbers $Z$ has three types of conjugations on $\mathbb{BC}$ \cite{bc_number}:
 \begin{align*}
    \overline{Z}=\overline{z}+i_2\overline{z}',\quad \tilde{Z}=z-i_2z',\quad
    Z^*=\overline{z}-i_2\overline{z}'.
\end{align*}
   If we consider standard inner products $\langle\cdot,\cdot\rangle_1 $ and $\langle\cdot,\cdot\rangle_2$ on two $\mathbb{C}(i_1)$ linear spaces $A$ and $B$ respectively. Then the standard inner product of two bicomplex numbers $Z$ and $W$ is given by $$ \langle Z,W\rangle = \langle z_1,w_1 \rangle_1 e_1+ \langle z_2,w_2 \rangle_2 e_2,$$ where $e_1$ and $e_2$ are defined in \eqref{eq:idm}.
 \begin{definition}\rm{\cite{bc_root}} Let $Z=z_1+i_2z_2\in {\mathbb{BC}}$. The Euclidean norm is defined as $$||Z||_2=\sqrt{|z|^2+|z'|^2}$$ and the hyperbolic norm is given by $$|Z|_h=|z_1|e_1+|z_2|e_2,$$ where $\{e_1,e_2\}$ is idempotent basis of the bicomplex module
$ \mathbb{BC}$. It can be observed that the hyperbolic norm of any bicomplex number is a hyperbolic number.\end{definition}
Using this hyperbolic norm open balls of bicomplex numbers can be defined as $$B_h(c,R)=\{Z:|Z-c|_h<_hR\},$$ where the center $c\in{\mathbb{BC}}$ and the hyperbolic radius $R\in{\mathbb{D}^+}$.\\
Let $f:\mathbb{U}\subset \mathbb{BC} \rightarrow \mathbb{BC}$ defined as $f(Z)=f_1(z_1) e_1+ f_2(z_2) e_2$, where $e_1$ and $e_2$ are defined in \eqref{eq:idm}. Then $f(Z)$ is said to be differentiable at $Z_0$ if $$\lim\limits_{\substack{{Z\to Z_0}\\{Z-Z_0\notin \mathbb{O}_2}}}\frac{f(Z)-f(Z_0)}{Z-Z_0}$$ exists finitely. Moreover, $f$ is said to be $\mathbb{BC}-holomorphic$ in $U$ if $f$ is differentiable at every points in $U$. \\
In \cite{bc_number}, analyticity condition is provided as follows:
\begin{theo}\rm{\cite{bc_number}}
     A function $F:\Omega\subset \mathbb{BC} \rightarrow \mathbb{BC}$, where  $F(Z)=f_1(z,z')+i_2f_2(z,z')$ is $\mathbb{BC}-holomorphic$ if and only if, $f_1$ and $f_2$ are holomorphic in $\Omega$ and satisfies following bicomplex Cauchy-Riemann equations :
      $$ \frac{\partial f_1}{\partial z}= \frac{\partial f_2}{\partial z'} \quad
     \mbox{and} \quad  \frac{\partial f_1}{\partial z'}=- \frac{\partial f_2}{\partial z}.$$
\end{theo}
\begin{definition}\rm{\cite{gamma}}
     The bicomplex gamma function is given by
    \begin{align*}
        \Gamma(Z)=\frac{e^{-\gamma Z}}{Z}\prod_{n=1}^{\infty}\left(\left(1+\frac{Z}{n}\right)^{-1}exp\left(\frac{Z}{n}\right)\right),\quad Z\in{\mathbb{BC}}
    \end{align*}
    where $Z=z_1+i_2z_2$ with $z_1\not=-\frac{(r+s)}{2}$ and $z_2\not=i_1\frac{(s-r)}{2}$ and  $r,s\in \mathbb{N}\cup \{0\}$. Moreover, $\Gamma(Z)$ can be expressed as follows :
    \begin{align}\label{eq:g}
        \Gamma(Z)=\Gamma(\alpha) e_1+ \Gamma(\beta) e_2,
    \end{align}
    where $e_1$ and $e_2$ are given by \eqref{eq:idm}.
\end{definition}

    A generalized hypergeometric function is defined \cite{sp_function} as:
    \begin{align}\label{eq:5}
         {}_{p} F_{q}\left[\begin{matrix}&a_1, &a_2,.... &a_{p};\\&b_1,&b_2,.....&b_
         {q};&\end{matrix}z\right]=
  \sum_{n=0}^{\infty}\frac{\prod\limits_{i=1}^{p}(a_i)_n}{\prod\limits_{j=1}^{q}(b_j)_n}\frac{z^n}{n!},
    \end{align}
    where $b_j$ is neither zero nor a negative integer.
\begin{theo}\rm{\cite{sp_function}}
    If $p\leqq q+1$, $\Re(b_1)>\Re(a_1)>0$ and $|z|<1$, then integral representation of the generalized hypergeometric function is given by
\begin{align}\label{eq:in}
    &{}_{p} F_{q}\left[\begin{matrix}&a_1, &a_2,.... &a_{p};\\&b_1,&b_2,.....&b_
         {q};&\end{matrix}z\right]\nonumber\\
         &= \frac{\Gamma(b_1)}{\Gamma(a_1)\Gamma(b_1-a_1)} \left[ \int_{0}^{1} t^{a_1-1}(1-t)^{b_1-a_1-1}{}_{p-1} F_{q-1}\left[\begin{matrix}&a_2,.... &a_{p};\\&b_2,.....&b_{q};&\end{matrix}(zt)\right] dt \right].
\end{align}
\end{theo}
\begin{theo}\rm{\cite{laplace transform}}\label{th:c1}
    If $p\leq q, \Re(v)>0$ and $|z|<1$, then another integral representation of the generalized hypergeometric function can be expressed as follows :
    \begin{align*}
         {}_{p+1} F_{q}\left[\begin{matrix}&a_0, &a_1,.... &a_{p};\\&b_1,&b_2,.....&b_
         {q};&\end{matrix}z\right]
         = \frac{1}{\Gamma(v)} \left[ \int_{0}^{\infty}e^{-t}t^{a_0-1}{}_{p} F_{q}\left[\begin{matrix}&a_1,.... &a_{p};\\&b_1,.....&b_{q};&\end{matrix}(zt)\right] dt \right].
    \end{align*}
\end{theo}
\begin{theo}[Saalschutz' theorem]\label{th:s}\rm{\cite{sp_function}} If $n$ is a non-negative integer and if $a, b, c$ are independent of $n$, then
\begin{align*}
    {}_{3} F_{2}\left[\begin{matrix}  & \quad \quad  -n, a, b;\\&c, 1-c+a+b-n;&\end{matrix}1\right] =\frac{(c-a)_n(c-b)_n}{(c)_n(c-a-b)_n}.
\end{align*}
\end{theo}
Special functions \cite{Das-2019,Das-Mehrez-2021,Das-Mehrez-2022,sp_function} play vital role in quantum physics, astrophysics and fluid dynamics. Many special functions such as the Gauss hypergeometric functions $({}_2F_{1})$, Confluent hypergeometric function $({}_1F_1)$ and Bessel function $(J_n(z))$ can be expressed as particular case of the generalized hypergeometric functions. Now a days researcher have focused on several special functions such as Gamma and Beta functions, Hypergeometric function, Mittag-Leffler function, Polygamma function, Riemman Zeta function, Hurwitz Zeta functions for the bicomplex argument \cite{gamma, bc_hypergeometric, polygamma, algebraic}. Motivated by the above results, we have studied generalized hypergeometric function for bicomplex argument.
\section{Bicomplex generalized hypergeometric function}
In this section, we introduce generalized hypergeometric function with bicomplex argument and discuss the convergence in bicomplex module. Let us define bicomplex generalized hypergeometric function as
\begin{align}\label{eq:gh}
{}_{p} F_{q}\left[\begin{matrix}&\alpha_1, &\alpha_2,\ldots, &\alpha_{p};\\&\beta_1,&\beta_2,\ldots, &\beta_
         {q};&\end{matrix}Z\right]=
  \sum_{n=0}^{\infty}\frac{\prod\limits_{i=1}^{p}(\alpha_i)_n}{\prod\limits_{j=1}^{q}(\beta_j)_n}\cdot \frac{Z^n}{n!},
\end{align}
where $Z=z_1e_1+z_2e_2,\;\alpha_i=\alpha_{1i}e_1+\alpha_{2i}e_2,\; \beta_j=\beta_{1j}e_1+\beta_{2j}e_2\in\mathbb{BC}$ for $i=1,2,\ldots,p$ and $j=1,2,\ldots,q$ with $\beta_{1j}$ and $\beta_{2j}$ are neither zero nor a negative integer and the series described in \eqref{eq:gh} is referred to as a bicomplex generalized hypergeometric series. The well definedness of the function \eqref{eq:gh} is also provided in the following theorem.
\begin{theo}
    Suppose that  $Z=z_1e_1+z_2e_2,\;\alpha_i=\alpha_{1i}e_1+\alpha_{2i}e_2,\; \beta_j=\beta_{1j}e_1+\beta_{2j}e_2 \in{\mathbb{BC}}$ for $i=1,2,\ldots,p$ and $j=1,2,\ldots,q$ with $\beta_{1j}$ and $\beta_{2j}$ are neither zero nor a negative integer and $\{e_1,e_2\}$ is idempotent basis of the bicomplex module
$ \mathbb{BC}$. Then idempotent representation of the bicomplex generalized  hypergeometric function is given by
 \begin{align}\label{eq:ab}
     {}_{p} F_{q}\left[\begin{matrix}&\alpha_{11}, &\alpha_{12},\ldots, &\alpha_{1p};\\&\beta_{11},&\beta_{12},\ldots,&\beta_{1q};&\end{matrix}z_1 \right]e_1 +{}_{p} F_{q}\left[\begin{matrix}&\alpha_{21}, &\alpha_{22},\ldots, &\alpha_{2p};\\&\beta_{21},&\beta_{22},\ldots,&\beta_{2q};&\end{matrix}z_2 \right]e_2.
 \end{align}
 \end{theo}
 \begin{proof}
 Form the definition of the Pochhammer symbol \cite{sp_function}, we have
\begin{align}\label{eq:fa}
    (\alpha_i)_n
    &=(\alpha_i)(\alpha_i+1)(\alpha_i+2)\ldots(\alpha_i+n-1) \nonumber\\
    &=(\alpha_{1i}e_1+\alpha_{2i}e_2)(\alpha_{1i}e_1+\alpha_{2i}e_2+1)(\alpha_{1i}e_1+\alpha_{2i}e_2+2)\ldots(\alpha_{1i}e_1+\alpha_{2i}e_2+n-1)\nonumber\\
    &=(\alpha_{1i}e_1+\alpha_{2i}e_2) \{(\alpha_{1i}+1)e_1+(\alpha_{2i}+1)e_2\} \nonumber\\& \;\;\;\times \{(\alpha_{1i}+2)e_1+(\alpha_{2i}+2)e_2\}\ldots \{(\alpha_{1i}+n-1)e_1+(\alpha_{2i}+n-1)e_2\}\nonumber\\
    &=(\alpha_{1i})(\alpha_{1i}+1)(\alpha_{1i}+2)\ldots(\alpha_{1i}+n-1)e_1+(\alpha_{2i})(\alpha_{2i}+1)(\alpha_{2i}+2)\ldots(\alpha_{2i}+n-1)e_2\nonumber\\
    &=(\alpha_{1i})_n e_1+(\alpha_{2i})_n e_2.
\end{align}
Similarly, we get
\begin{align}\label{eq:fb}
    (\beta_i)_n =(\beta_{1j})_n e_1+(\beta_{2j})_n e_2.
\end{align}
Using \eqref{eq:fa}, \eqref{eq:fb} and \eqref{eq:gh}, we obtain
\begin{align*}
   {}_{p} F_{q}\left[\begin{matrix}&\alpha_1, &\alpha_2,\ldots,&\alpha_{p};\\&\beta_1,&\beta_2,\ldots,&\beta_{q};&\end{matrix}Z\right]
   &=\sum_{n=0}^{\infty}\frac{\prod\limits_{i=1}^{p}(\alpha_{1i}e_1+\alpha_{2i}e_2)_n}{\prod\limits_{j=1}^{q}(\beta_{1j}e_1+\beta_{2j}e_2)_n}\cdot\frac{(z_1^{n} e_1+z_2^{n}e_2)}{n!}\nonumber\\
  & =\sum_{n=0}^{\infty}\frac{\prod\limits_{i=1}^{p}(\alpha_{1i})_n}{\prod\limits_{j=1}^{q}(\beta_{1j})_n}\cdot\frac{z_1^n}{n!} e_1+\sum_{n=0}^{\infty}\frac{\prod\limits_{i=1}^{p}(\alpha_{2i})_n}{\prod\limits_{j=1}^{q}(\beta_{2j})_n}\cdot\frac{z_2^n}{n!} e_2\nonumber\\
  &={}_{p} F_{q}\left[\begin{matrix}&\alpha_{11}, &\alpha_{12},\ldots, &\alpha_{1p};\\&\beta_{11},&\beta_{12},\ldots,&\beta_{1q};&\end{matrix}z_1 \right]e_1 +{}_{p} F_{q}\left[\begin{matrix}&\alpha_{21}, &\alpha_{22},\ldots, &\alpha_{2p};\\&\beta_{21},&\beta_{22},\ldots,&\beta_{2q};&\end{matrix}z_2 \right]e_2.
\end{align*}
Now, $\beta_{1j}$ and $\beta_{2j}$ are neither zero nor a negative integer for $i=1,2,\ldots,p$ and $j=1,2,\ldots,q$ so that the function  ${}_{p} F_{q}\left[\begin{matrix}&\alpha_{11}, &\alpha_{12},\ldots, &\alpha_{1p};\\&\beta_{11},&\beta_{12},\ldots,&\beta_{1q};&\end{matrix}z_1 \right]$ and ${}_{p} F_{q}\left[\begin{matrix}&\alpha_{21}, &\alpha_{22},\ldots, &\alpha_{2p};\\&\beta_{21},&\beta_{22},\ldots,&\beta_{2q};&\end{matrix}z_2 \right]$ are well defined. Which completes the proof of the theorem.
\end{proof}
\begin{theo}
    Assume that  $Z=z_1e_1+z_2e_2,\;\alpha_i=a_{1i}+i_2a_{2i}=\alpha_{1i}e_1+\alpha_{2i}e_2$ and $\beta_j=b_{1j}+i_2b_{2j}=\beta_{1j}e_1+\beta_{2j}e_2 \in{\mathbb{BC}}$ for $i=1,2,\ldots,p$ and $j=1,2,\ldots,q$ with $\beta_{1j}$ and $\beta_{2j}$ are neither zero nor a negative integer and $\{e_1,e_2\}$ is idempotent basis of the bicomplex module
$ \mathbb{BC}$. Then convergence condition of the bicomplex generalized hypergeometric series\eqref{eq:gh} as follows :
\begin{description}
    \item[(a)] If $p\leq q$, then the series is absolutely hyperbolic convergent for all $Z\in \mathbb{BC}$.
    \item[(b)] If $p=q+1$, then the series hyperbolically converges absolutely in the ball $\mathbb{B}_h(0,1)$ and diverges on the complement of its closure. Moreover, the series hyperbolic uniformly and absolutely convergent on the boundary of the ball $\mathbb{B}_h(0,1)$ if
\begin{align}\label{eq:asf}
    \Re\left(\sum\limits_{j=1}^{q}b_{1j}-\sum\limits_{i=1}^{p}a_{1i}\right)>\left|\Im\left(\sum\limits_{j=1}^{q}b_{2j}-\sum\limits_{i=1}^{p}a_{2i}\right)\right|.
\end{align}
    \item[(c)] If $p>q+1$, then the series diverges for all nonzero $Z\in \mathbb{BC}$.
\end{description}
\end{theo}
\begin{proof}
Let us consider
\begin{align}\label{eq:asd}
     {}_{p} F_{q}\left[\begin{matrix}&\alpha_1, &\alpha_2,\ldots, &\alpha_{p};\\&\beta_1,&\beta_2,\ldots,&\beta_
         {q};&\end{matrix}Z\right]=
  \sum_{n=0}^{\infty}a_n {Z^n} =\left(\sum_{n=0}^{\infty}a_{1n} z_1^{n}\right)e_1+\left(\sum_{n=0}^{\infty}a_{2n} z_2^{n}\right)e_2,
\end{align}
where $a_{1n}=\frac{\prod\limits_{i=1}^{p}(\alpha_{1i})_n}{\prod\limits_{j=1}^{q}(\beta_{1j})_n}\frac{1}{n!} $ and $a_{2n}=\frac{\prod\limits_{i=1}^{p}(\alpha_{2i})_n}{\prod\limits_{j=1}^{q}(\beta_{2j})_n}\frac{1}{n!} $. We will now check convergence of the bicomplex generalized hypergeometric function using hyperbolic ratio test \cite{bc_root}. Let $R=R_1e_1+R_2e_2$ be the radius of convergence of the bicomplex power series \eqref{eq:asd}. Then
\begin{align}\label{eq:su}
    R&=R_1e_1+R_2e_2=\lim_{n\to \infty} \sup\frac{|a_n|_h}{|a_{n+1}|_h}\nonumber\\
    &=\lim_{n\to \infty} \sup\frac{|a_{1n}|}{|a_{1n+1}|}e_1+\lim_{n\to \infty} \sup\frac{|a_{2n}|}{|a_{2n+1}|}e_2\nonumber\\
    &=\lim_{n\to \infty} \sup\left|\frac{\prod\limits_{j=1}^{q}(\beta_{1j}+n)(n+1)}{\prod\limits_{i=1}^{p}(\alpha_{1i}+n)}\right | e_1+ \lim_{n\to \infty} \sup\left|\frac{\prod\limits_{j=1}^{q}(\beta_{2j}+n)(n+1)}{\prod\limits_{i=1}^{p}(\alpha_{2i}+n)}\right | e_2\nonumber\\
    &=\lim_{n\to \infty} {\sup} \left|\frac{\prod\limits_{j=1}^{q}(\frac{\beta_{1j}}{n}+1)(1+\frac{1}{n})}{\prod\limits_{i=1}^{p}(\frac{\alpha_{1i}}{n}+1)\times n^{p-1-q}}\right | e_1+ \lim_{n\to \infty} \sup\left|\frac{\prod\limits_{j=1}^{q}(\frac{\beta_{2j}}{n}+1)(1+\frac{1}{n})}{\prod\limits_{i=1}^{p}(\frac{\alpha_{2i}}{n}+1)\times n^{p-1-q}}\right | e_2.
\end{align}
Using hyperbolic ratio test \cite{bc_root} and \eqref{eq:su}, we obtain
\begin{description}
    \item[(a)] If $p\leq q$, then $R_1=\infty$ and $R_2=\infty$, which implies that $R=\infty$, so the series \eqref{eq:gh} is absolutely hyperbolic convergent for all $Z\in \mathbb{BC}$.
    \item[(b)] If $p=q+1$, then $R_1=1$ and $R_2=1$, that yields $R=e_1+e_2=1$, the series \eqref{eq:gh} hyperbolically converges absolutely in the ball $\mathbb{B}_h(0,1)$ and diverges on the complement of its closure. The ball $\mathbb{B}_h(0,1)$ is given by cartesian product of the two disk $B_{e_1}\left(0,\frac{1}{\sqrt{2}}\right)$  and $B_{e_2}\left(0,\frac{1}{\sqrt{2}}\right)$ are situated in the plane $\mathbb{BC}_{e_1}=\{ze_1:z\in{\mathbb{C}}\}$ and $\mathbb{BC}_{e_2}=\{ze_2:z\in{\mathbb{C}}\}$ respectively (see figure \ref{fig:sd}). Though the figure \ref{fig:sd}(c) is indeed of four dimensional object but for abstract visualization it is represented by the interior of the three dimensional object.
    \item[(c)] If $p>q+1$, then $R=0$. Hence the series \eqref{eq:gh} diverges for all nonzero $Z\in \mathbb{BC}$.
\end{description}
\begin{figure}[h!]
  \centering
  \begin{subfigure}{0.28\linewidth}
    \includegraphics[width=\linewidth]{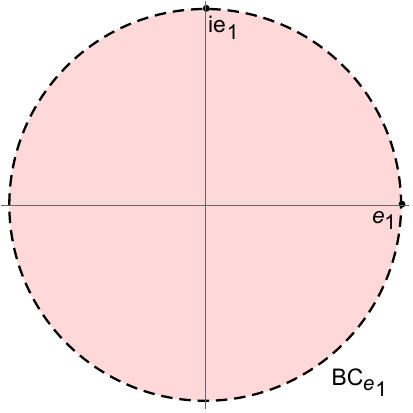}
    \caption{Disk $B_{e_1}$.}
  \end{subfigure}
  \begin{subfigure}{0.28\linewidth}
    \includegraphics[width=\linewidth]{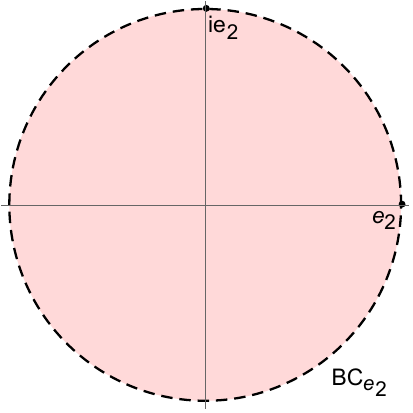}
    \caption{Disk $B_{e_2}$.}
  \end{subfigure}
  \begin{subfigure}{0.38\linewidth}
    \includegraphics[width=\linewidth]{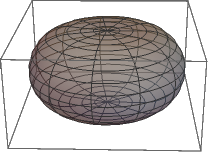}
    \caption{Outer surface of $B_{e_1}\times B_{e_2}$.}
  \end{subfigure}
  \caption{}
  \label{fig:sd}
\end{figure}
Now, we will check the convergence of the series \eqref{eq:gh} on the boundary of the ball $\mathbb{B}_h(0,1)$ for the condition $p=q+1$.
In this case, $|Z|_h=1$ which implies $|z_1|=1$ and $|z_2|=1$. From given inequality \eqref{eq:asf} we have
\begin{align}\label{eq:jn}
     \Re\left(\sum\limits_{j=1}^{q}b_{1j}-\sum\limits_{i=1}^{p}a_{1i}\right)+\Im\left(\sum\limits_{j=1}^{q}b_{2j}-\sum\limits_{i=1}^{p}a_{2i}\right)>0
\end{align} and
\begin{align}\label{eq:cr}
     \Re\left(\sum\limits_{j=1}^{q}b_{1j}-\sum\limits_{i=1}^{p}a_{1i}\right)-\Im\left(\sum\limits_{j=1}^{q}b_{2j}-\sum\limits_{i=1}^{p}a_{2i}\right)>0.
\end{align}
Since $\alpha_i=a_{1i}+i_2a_{2i}$ and $\beta_j=b_{1j}+i_2b_{2j}$ then from inequality \eqref{eq:jn}, we get
\begin{align*}
     \Re\left(\sum\limits_{j=1}^{q}(b_{1j}-i_1b_{2j})-\sum\limits_{i=1}^{p}(a_{1i}-i_1a_{2i})\right)>0 \implies \Re(\sum\limits_{j=1}^{q}\beta_{1j}-\sum\limits_{i=1}^{p} \alpha_{1i})>0.
\end{align*}
Similarly, from other inequality \eqref{eq:cr}, we obtain
$$\Re(\sum\limits_{j=1}^{q}\beta_{2j}-\sum\limits_{i=1}^{p} \alpha_{2i})>0.$$
We assume that $\eta_1=\Re\left(\sum\limits_{j=1}^{q}\beta_{1j}-\sum\limits_{i=1}^{p} \alpha_{1i}\right)>0$ and $\eta_2=\Re(\sum\limits_{j=1}^{q}\beta_{2j}-\sum\limits_{i=1}^{p} \alpha_{2i})>0$.\\
Now,
\begin{align}
    |a_n|_h&=|a_{1n}|e_1 +|a_{2n}| e_2 \nonumber\\
      &=\left|\frac{\prod\limits_{i=1}^{p}(\alpha_{1i})_n}{\prod\limits_{j=1}^{q}(\beta_{1j})_n}.\frac{1}{n!}\right| e_1+\left|\frac{\prod\limits_{i=1}^{p}(\alpha_{2i})_n}{\prod\limits_{j=1}^{q}(\beta_{2j})_n}.\frac{1}{n!}\right| e_2\nonumber\\
      &=\left|\frac{\prod\limits_{i=1}^{p}\frac{(\alpha_{1i})_n}{(n-1)!\cdot n^{\alpha_{1i}}}}{\prod\limits_{j=1}^{q}\frac{(\beta_{1j})_n}{(n-1)!\cdot n^{\beta_{1j}}}}\times \frac{1}{n^{\left(\sum\limits_{j=1}^{q}\beta_{1j}-\sum\limits_{i=1}^{p} \alpha_{1i}+1\right)}}\right|e_1+\left|\frac{\prod\limits_{i=1}^{p}\frac{(\alpha_{2i})_n}{(n-1)!\cdot n^{\alpha_{2i}}}}{\prod\limits_{j=1}^{q}\frac{(\beta_{2j})_n}{(n-1)!\cdot n^{\beta_{2j}}}}\times \frac{1}{n^{\left(\sum\limits_{j=1}^{q}\beta_{2j}-\sum\limits_{i=1}^{p} \alpha_{2i}+1\right)}}\right|e_2\nonumber\\
      &= \left|\frac{\prod\limits_{i=1}^{p}\frac{1}{\Gamma(\alpha_{1i})}}{\prod\limits_{j=1}^{q}\frac{1}{\Gamma(\beta_{1j})}}\times \frac{1}{n^{\left(\sum\limits_{j=1}^{q}\beta_{1j}-\sum\limits_{i=1}^{p} \alpha_{1i}+1\right)}} \right|e_1+ \left|\frac{\prod\limits_{i=1}^{p}\frac{1}{\Gamma(\alpha_{2i})}}{\prod\limits_{j=1}^{q}\frac{1}{\Gamma(\beta_{2j})}}\times \frac{1}{n^{\left(\sum\limits_{j=1}^{q}\beta_{2j}-\sum\limits_{i=1}^{p} \alpha_{2i}+1\right)}}\right|e_2 \nonumber\\
      &=\left|\frac{\prod\limits_{i=1}^{p}\frac{1}{\Gamma(\alpha_{1i})}}{\prod\limits_{j=1}^{q}\frac{1}{\Gamma(\beta_{1j})}}\right|\times \frac{e_1}{n^{\eta_1+1}}+ \left|\frac{\prod\limits_{i=1}^{p}\frac{1}{\Gamma(\alpha_{2i})}}{\prod\limits_{j=1}^{q}\frac{1}{\Gamma(\beta_{2j})}}\right|\times \frac{e_2}{n^{\eta_2+1}}.\nonumber
\end{align}
Now we consider $M_n=\left|\frac{\prod\limits_{i=1}^{p}\frac{1}{\Gamma(\alpha_{1i})}}{\prod\limits_{j=1}^{q}\frac{1}{\Gamma(\beta_{1j})}}\right|\times \frac{e_1}{n^{\eta_1+1}}+ \left|\frac{\prod\limits_{i=1}^{p}\frac{1}{\Gamma(\alpha_{2i})}}{\prod\limits_{j=1}^{q}\frac{1}{\Gamma(\beta_{2j})}}\right|\times \frac{e_2}{n^{\eta_2+1}} $. Since, the series $\sum M_n$ is hyperbolic convergent, then by
hyperbolic Weierstrass M-test  \cite{bc_root}, we get the series $\sum a_n Z^{n}$ is hyperbolic uniformly and absolutely convergent.
The proof is now completed.
\end{proof}
\begin{corollary}
    Setting $p=1$ and $q=1$ in \eqref{eq:gh}, we get bicomplex confluent hypergeometric function as follows
    \begin{align*}
        {}_1F_1(\alpha_1;\beta_1;Z)=\sum\limits_{n=0}^{\infty}\frac{(\alpha_1)_n}{(\beta_1)_n}.\frac{Z^n}{n!},
    \end{align*}
    where $\alpha_1=\alpha_{11}e_1+\alpha_{21}e_2,\;\beta_1=\beta_{11}e_1+\beta_{21}e_2$, $Z\in \mathbb{BC}$ with $\beta_{11}$ and $\beta_{21}$ are neither zero nor a negative integer.
\end{corollary}
\begin{corollary}
    Substituting of $p=2$ and $q=1$ in \eqref{eq:gh}, yields the bicomplex hypergeometric function as
 \begin{align*}
{}_2F_1(\alpha_1,\alpha_2;\beta_1;Z)=\sum\limits_{n=0}^{\infty}\frac{(\alpha_1)_n(\alpha_2)_n}{(\beta_1)_n}.\frac{Z^n}{n!},
\end{align*}
where $Z\in\mathbb{B}(0,1)$, $\alpha_1=\alpha_{11}e_1+\alpha_{21}e_2,\;\alpha_2=\alpha_{12}e_1+\alpha_{22}e_2, \;\beta_1=\beta_{11}e_1+\beta_{21}e_2\in \mathbb{BC}$ with $\beta_{11}$ and $\beta_{21}$ are neither zero nor negative integer and $\{e_1, e_2\}$ is idempotent basis of the bicomplex module $\mathbb{BC}$.
\end{corollary}

     \section{ Integral representations of the bicomplex generalized hypergeometric function}
 In this section, we establish three different types of integral representations of the bicomplex generalized hypergeometric function.
  \begin{theo}\label{th:ru}
     Suppose that $Z\in\mathbb{B}_h(0,1)$, $\Re(b_{11}-a_{11})>|\Im(b_{21}-a_{21})|$ and $\Re(a_{11})>|\Im(a_{21})|$. The integral representation of bicomplex generalized hypergeometric function is given by
     \begin{align*}
         & {}_{p} F_{q}\left[\begin{matrix}&\alpha_1, &\alpha_2,\ldots, &\alpha_{p};\\&\beta_1,&\beta_2,\ldots,&\beta_
         {q};&\end{matrix}Z\right]
         =\frac{\Gamma(\beta_{1})}{\Gamma(\alpha_{1})\Gamma(\beta_{1}-\alpha_{1})} \left[\sum_{n=0}^{\infty} \frac{\prod\limits_{i=2}^{p} (\alpha_{i})_n}{\prod\limits_{j=2}^{q} (\beta_{j})_n}\cdot\frac{Z^n}{n!}\int_{C} t^{\alpha_{1}+n-1}(1-t)^{\beta_{1}-\alpha_{1}-1}dt\right],
     \end{align*}
    where $\alpha_i=a_{1i}+i_2a_{2i}=\alpha_{1i}e_1+\alpha_{2i}e_2,\; \beta_j=b_{1j}+i_2b_{2j}=\beta_{1j}e_1+\beta_{2j}e_2,\;t=t_1e_1+t_2e_2 \in{\mathbb{BC}}$ for $i=1,2,\ldots,p$ and $j=1,2,\ldots,q$ with $\beta_{1j}$ and $\beta_{2j}$ are neither zero nor a negative integer and $C(t)=\left(C_1(t_1),C_2(t_2)\right)$ be a curve in $\mathbb{BC}$ with $0\leq t_1,t_2\leq 1$.
 \end{theo}
 \begin{proof}
 Using \eqref{eq:ab}, \eqref{eq:g} and \eqref{eq:in}, we get
 \begin{align*}
    & {}_{p} F_{q}\left[\begin{matrix}&\alpha_1, &\alpha_2,\ldots, &\alpha_{p};\\&\beta_1,&\beta_2,\ldots,&\beta_
         {q};&\end{matrix}Z\right]\\&={}_{p} F_{q}\left[\begin{matrix}&\alpha_{11}, &\alpha_{12},\ldots, &\alpha_{1p};\\&\beta_{11},&\beta_{12},\ldots,&\beta_{1q};&\end{matrix}z_1 \right]e_1 +{}_{p} F_{q}\left[\begin{matrix}&\alpha_{21}, &\alpha_{22},\ldots, &\alpha_{2p};\\&\beta_{21},&\beta_{22},\ldots,&\beta_{2q};&\end{matrix}z_2 \right]e_2\\
         &= \frac{\Gamma(\beta_{11})}{\Gamma(\alpha_{11})\Gamma(\beta_{11}-\alpha_{11})} \left[ \int_{0}^1 {t_1}^{\alpha_{11}-1}(1-t_1)^{\beta_{11}-\alpha_{11}-1}{}_{p-1} F_{q-1}\left[\begin{matrix}&\alpha_{12},\ldots, &\alpha_{1p};\\&\beta_{12},\ldots,&\beta_{1q};&\end{matrix}(z_1 t_1)\right] dt_1 \right]e_1\\
         &\quad+\frac{\Gamma(\beta_{21})}{\Gamma(\alpha_{21})\Gamma(\beta_{21}-\alpha_{21})} \left[ \int_{0}^1 {t_2}^{\alpha_{21}-1}(1-t_2)^{\beta_{21}-\alpha_{21}-1}{}_{p-1} F_{q-1}\left[\begin{matrix}&\alpha_{22},\ldots, &\alpha_{2p};\\&\beta_{22},\ldots,&\beta_{2q};&\end{matrix}(z_2t_2)\right] dt_2 \right]e_2\\
         &=\frac{\Gamma(\beta_{11})}{\Gamma(\alpha_{11})\Gamma(\beta_{11}-\alpha_{11})} \left[ \int_{C_1} {t_1}^{\alpha_{11}-1}(1-t_1)^{\beta_{11}-\alpha_{11}-1}\sum_{n=0}^{\infty} \frac{\prod\limits_{i=2}^{p} (\alpha_{1i})_n}{\prod\limits_{j=2}^{q} (\beta_{1j})_n}\cdot\frac{(z_1t_1)^n}{n!} dt_1 \right]e_1\\
         &\quad+\frac{\Gamma(\beta_{21})}{\Gamma(\alpha_{21})\Gamma(\beta_{21}-\alpha_{21})} \left[ \int_{C_2} {t_2}^{\alpha_{21}-1}(1-t_2)^{\beta_{21}-\alpha_{21}-1}\sum_{n=0}^{\infty} \frac{\prod\limits_{i=2}^{p} (\alpha_{2i})_n}{\prod\limits_{j=2}^{q} (\beta_{2j})_n}\cdot\frac{(z_2t_2)^n}{n!} dt_2 \right]e_2\\
         &=\frac{\Gamma(\beta_{1})}{\Gamma(\alpha_{1})\Gamma(\beta_{1}-\alpha_{1})} \left[\int_{C} t^{\alpha_{1}-1}(1-t)^{\beta_{1}-\alpha_{1}-1}\sum_{n=0}^{\infty} \frac{\prod\limits_{i=2}^{p} (\alpha_{i})_n}{\prod\limits_{j=2}^{q} (\beta_{j})_n}\cdot\frac{(Zt)^n}{n!} dt \right]\\
         &=\frac{\Gamma(\beta_{1})}{\Gamma(\alpha_{1})\Gamma(\beta_{1}-\alpha_{1})}\left[\sum_{n=0}^{\infty} \frac{\prod\limits_{i=2}^{p} (\alpha_{i})_n}{\prod\limits_{j=2}^{q} (\beta_{j})_n}\cdot\frac{Z^n}{n!}\int_{C} t^{\alpha_{1}+n-1}(1-t)^{\beta_{1}-\alpha_{1}-1}dt\right].
 \end{align*}
  Hence the proof is completed.
 \end{proof}
 \begin{corollary}\label{co:er}
 The integral representation of bicomplex confluent hypergeometric function is given by
 \begin{align*}
     {}_1F_1(\alpha_1;\beta_1;Z)=\frac{\Gamma(\beta_{1})}{\Gamma(\alpha_{1})\Gamma(\beta_{1}-\alpha_{1})}\int_{C} t^{\alpha_{1}-1}(1-t)^{\beta_{1}-\alpha_{1}-1}e^{(Zt)}dt,
 \end{align*}
 where $Z\in\mathbb{B}_h(0,1),\; t=t_1e_1+t_2e_2\in\mathbb{BC}$  and $C(t)=\left(C_1(t_1),C_2(t_2)\right)$ be a curve in $\mathbb{BC}$ with $0\leq t_1,t_2\leq 1$, provided the conditions $\Re(b_{11}-a_{11})>|\Im(b_{21}-a_{21})|$ and $\Re(a_{11})>|\Im(a_{21})|$ hold.
  \end{corollary}
 \begin{proof}
      Setting $p=q=1$ in Theorem \ref{th:ru}, we obtain integral representation of the bicomplex confluent hypergeometric function as follows :
    \begin{align*}
        {}_1F_1(\alpha_1;\beta_1;Z)&=\frac{\Gamma(\beta_{1})}{\Gamma(\alpha_{1})\Gamma(\beta_{1}-\alpha_{1})}\left[\sum_{n=0}^{\infty}\frac{Z^n}{n!}\int_{C} t^{\alpha_{1}+n-1}(1-t)^{\beta_{1}-\alpha_{1}-1}dt\right]\\
        &=\frac{\Gamma(\beta_{1})}{\Gamma(\alpha_{1})\Gamma(\beta_{1}-\alpha_{1})}\int_{C} t^{\alpha_{1}-1}(1-t)^{\beta_{1}-\alpha_{1}-1}\sum_{n=0}^\infty\frac{(Zt)^n}{n!}dt\\
       &=\frac{\Gamma(\beta_{1})}{\Gamma(\alpha_{1})\Gamma(\beta_{1}-\alpha_{1})}\int_{C} t^{\alpha_{1}-1}(1-t)^{\beta_{1}-\alpha_{1}-1}e^{(Zt)}dt.
        \end{align*}
 \end{proof}
 \begin{example}
     The integral representation of $\frac{1}{Z^2}(e^Z-1-Z)$ is given by
     \begin{align*}
          \frac{1}{Z^2}\left(e^Z-1-Z\right)=\frac{1}{2}{}_1F_1(1;3;Z)=\int_{C} (1-t)e^{(Zt)}dt,
     \end{align*}
    where $t=t_1e_1+t_2e_2\in\mathbb{BC}$ and  $C(t)=\left(C_1(t_1),C_2(t_2)\right)$ be a curve in $\mathbb{BC}$ with $0\leq t_1,t_2\leq 1$.
 \end{example}
\begin{corollary}\label{co:qs}
 The integral representation of bicomplex hypergeometric function is
 \begin{align*}
     {}_2F_1(\alpha_1,\alpha_2;\beta_1;Z)=\frac{\Gamma(\beta_{1})}{\Gamma(\alpha_{1})\Gamma(\beta_{1}-\alpha_{1})}\left[\int_{C} t^{\alpha_{1}-1}(1-t)^{\beta_{1}-\alpha_{1}-1}(1-Zt)^{-\alpha_2}dt\right],
 \end{align*}
 where  $Z\in\mathbb{B}_h(0,1),\; t=t_1e_1+t_2e_2\in \mathbb{BC}$ and $C(t)=\left(C_1(t_1),C_2(t_2)\right)$ be a curve in $\mathbb{BC}$ with $0\leq t_1,t_2\leq 1$ , provided the conditions $\Re(b_{11}-a_{11})>|\Im(b_{21}-a_{21})|$ and $\Re(a_{11})>|\Im(a_{21})|$.
 \end{corollary}
 \begin{proof}
     Setting $p=2$ and $q=1$ in Theorem \ref{th:ru}, we get
      \begin{align*}
          {}_2F_1(\alpha_1,\alpha_2;\beta_1;Z)&=\frac{\Gamma(\beta_{1})}{\Gamma(\alpha_{1})\Gamma(\beta_{1}-\alpha_{1})}\left[\sum_{n=0}^{\infty}\frac{(\alpha_2)_nZ^n}{n!}\int_{C} t^{\alpha_{1}+n-1}(1-t)^{\beta_{1}-\alpha_{1}-1}dt\right]\\
          &=\frac{\Gamma(\beta_{1})}{\Gamma(\alpha_{1})\Gamma(\beta_{1}-\alpha_{1})}\left[\int_{C} t^{\alpha_{1}-1}(1-t)^{\beta_{1}-\alpha_{1}-1}\sum_{n=0}^{\infty}\frac{(\alpha_2)_n(Zt)^n}{n!}dt\right]\\
          &=\frac{\Gamma(\beta_{1})}{\Gamma(\alpha_{1})\Gamma(\beta_{1}-\alpha_{1})}\left[\int_{C} t^{\alpha_{1}-1}(1-t)^{\beta_{1}-\alpha_{1}-1}\sum_{n=0}^{\infty}\frac{(\alpha_2)_n(Zt)^n}{n!}dt\right].
      \end{align*}
     From \cite[page 142]{gamma}, we have
       $$(1+w)^a=\sum_{n=0}^{\infty}\frac{\Gamma(a+1)}{\Gamma(a-n+1)}\frac{w^n}{n!},$$ where $a,w\in\mathbb{BC}$. Now, using binomial theorem for bicomplex arguments and idempotent representation of a bicomplex series, we get
      \begin{align*}
          &\sum_{n=0}^{\infty}\frac{(\alpha_2)_n(Zt)^n}{n!}\\&=\sum_{n=0}^{\infty}\frac{(\alpha_{12})_n(z_1t_1)^n}{n!}e_1+\sum_{n=0}^{\infty}\frac{(\alpha_{22})_n(z_2t_2)^n}{n!}e_2\\
          &=\sum_{n=0}^{\infty}\frac{\alpha_{12}(\alpha_{12}+1)\ldots(\alpha_{12}+n-1)(z_1t_1)^n}{n!}e_1+\sum_{n=0}^{\infty}\frac{\alpha_{22}(\alpha_{22}+1)\ldots (\alpha_{22}+n-1)(z_2t_2)^n}{n!}e_2\\
          &=\sum_{n=0}^{\infty}\frac{(-\alpha_{12})(-\alpha_{12}-1)\ldots(-\alpha_{12}-n+1)(-z_1t_1)^n}{n!}e_1\\&\quad+\sum_{n=0}^{\infty}\frac{(-\alpha_{22})(-\alpha_{22}-1)\ldots (-\alpha_{22}-n+1)(-z_2t_2)^n}{n!}e_2\\
          &=\sum_{n=0}^{\infty}\frac{\Gamma(-\alpha_{12}+1)}{\Gamma(-\alpha_{12}-n+1)} \frac{(-z_1t_1)^n}{n!}e_1+\sum_{n=0}^{\infty}\frac{\Gamma(-\alpha_{22}+1)}{\Gamma(-\alpha_{22}-n+1)} \frac{(-z_2t_2)^n}{n!}e_2\\
          &=\sum_{n=0}^{\infty}\frac{\Gamma(-\alpha_{2}+1)}{\Gamma(-\alpha_{2}-n+1)} \frac{(-Zt)^n}{n!}=(1-Zt)^{-\alpha_2}.
      \end{align*}
      Hence, the integral representation of the bicomplex hypergeometric function is
      \begin{align*}
           {}_2F_1(\alpha_1,\alpha_2;\beta_1;Z)=\frac{\Gamma(\beta_{1})}{\Gamma(\alpha_{1})\Gamma(\beta_{1}-\alpha_{1})}\left[\int_{C} t^{\alpha_{1}-1}(1-t)^{\beta_{1}-\alpha_{1}-1}(1-Zt)^{-\alpha_2}dt\right].
      \end{align*}
  \end{proof}
  \begin{example}
      The integral representation of $\frac{1}{(1-Z)^2}$ is given by
      \begin{align*}
           \frac{1}{(1-Z)^2}={}_2F_1(1,2;1;Z)=\left[\int_{C}(1-Zt)^{-2}dt\right],
      \end{align*}
      where $Z\in\mathbb{B}_h(0,1),\; t=t_1e_1+t_2e_2\in \mathbb{BC}$ and $C$ be a curve in $\mathbb{BC}$ whose parametric representation is $C(t)=\left(C_1(t_1),C_2(t_2)\right)$ with $0\leq t_1,t_2\leq 1$.
  \end{example}
  \begin{corollary}
  If  $\Re(b_{11})>\Re(a_{11})>0$ and $|z|<1$, then confluent hypergeometric function can be expressed as follows :
  \begin{align*}
      {}_1F_1(a_{11};b_{11};z)=\frac{\Gamma(b_{11})}{\Gamma(a_{11})\Gamma(b_{11}-a_{11})}\left[\int_{0}^1 u^{a_{11}-1}(1-u)^{b_{11}-a_{11}-1}e^{(zu)}du\right].
  \end{align*}
  \end{corollary}
  \begin{proof}
       Putting $a_{21}=b_{21}=0$ and $Z=z\in\mathbb{C}$  in Corollary \ref{co:er}, we have
     \begin{align*}
         &{}_1F_1(a_{11};b_{11};z)\\&=\frac{\Gamma(b_{11})}{\Gamma(a_{11})\Gamma(b_{11}-a_{11})}\int_{C} t^{a_{11}-1}(1-t)^{b_{11}-a_{11}-1}e^{(zt)}dt\\
         &=\frac{\Gamma(b_{11})}{\Gamma(a_{11})\Gamma(b_{11}-a_{11})}\left[\int_{0}^1 t_1^{a_{11}-1}(1-t)^{b_{11}-a_{11}-1}e^{(zt_1)}dt_1e_1+\int_{0}^1 t_2^{a_{11}-1}(1-t_2)^{b_{11}-a_{11}-1}e^{(zt_2)}dt_2e_2\right]\\
         &=\frac{\Gamma(b_{11})}{\Gamma(a_{11})\Gamma(b_{11}-a_{11})}\left[\int_{0}^1 u^{a_{11}-1}(1-u)^{b_{11}-a_{11}-1}e^{(zu)}du(e_1+e_2)\right]\\
         &=\frac{\Gamma(b_{11})}{\Gamma(a_{11})\Gamma(b_{11}-a_{11})}\left[\int_{0}^1 u^{a_{11}-1}(1-u)^{b_{11}-a_{11}-1}e^{(zu)}du\right].
     \end{align*}
\end{proof}
 \begin{corollary}
 If  $\Re(b_{11})>\Re(a_{11})>0$ and $|z|<1$, the integral representation of confluent hypergeometric function is
 \begin{align*}
     {}_2F_1(a_{11},a_{12};b_{11};z)=\frac{\Gamma(b_{11})}{\Gamma(a_{11})\Gamma(b_{11}-a_{11})}\left[\int_{0}^1 \frac{u^{a_{11}-1}(1-u)^{b_{11}-a_{11}-1}}{(1-zu)^{a_{21}}}du\right].
 \end{align*}
 \end{corollary}
 \begin{proof}
     Putting $a_{12}=a_{22}=b_{21}=0$ and $Z=z\in \mathbb{C}$ in Corollary \ref{co:qs}, we get
     \begin{align*}
          &{}_2F_1(a_{11},a_{12};b_{11};z)\\&=\frac{\Gamma(b_{11})}{\Gamma(a_{11})\Gamma(b_{11}-a_{11})}\int_{C} t^{a_{11}-1}(1-t)^{b_{11}-a_{11}-1}(1-zt)^{-a_{21}}dt\\
          &=\frac{\Gamma(b_{11})}{\Gamma(a_{11})\Gamma(b_{11}-a_{11})}\left[\int_{0}^1 \frac{t_1^{a_{11}-1}(1-t_1)^{b_{11}-a_{11}-1}}{(1-zt_1)^{a_{21}}}dte_1
          +\int_{0}^1 \frac{t_2^{a_{11}-1}(1-t_2)^{b_{11}-a_{11}-1}}{(1-zt_2)^{a_{21}}}dte_2\right]\\
          &=\frac{\Gamma(b_{11})}{\Gamma(a_{11})\Gamma(b_{11}-a_{11})}\left[\int_{0}^1 \frac{u^{a_{11}-1}(1-u)^{b_{11}-a_{11}-1}}{(1-zu)^{a_{21}}}du(e_1+e_2)\right]\\
          &=\frac{\Gamma(b_{11})}{\Gamma(a_{11})\Gamma(b_{11}-a_{11})}\left[\int_{0}^1 \frac{u^{a_{11}-1}(1-u)^{b_{11}-a_{11}-1}}{(1-zu)^{a_{21}}}du\right],
     \end{align*}
     which completes the proof.
\end{proof}
  \begin{theo}\label{th:mn}
      If $p\leq q,\;Z\in\mathbb{B}_h(0,1)$ and $v=v_1+j v_2=v_{11}e_1+v_{12}e_2\in\mathbb{BC}$  with $\Re(v_1)>|\Im(v_2)|$, the integral representation of the bicomplex generalized hypergeometric function is given by
      \begin{align*}
          {}_{p+1} F_{q}\left[\begin{matrix}&v, &a_1,\ldots, &a_{p};\\&b_1,&b_2,\ldots,&b_
         {q};&\end{matrix}Z\right]
         = \frac{1}{\Gamma(v)}\left[ \sum_{n=0}^{\infty}\frac{\prod\limits_{i=1}^{p} (\alpha_{i})_n}{\prod\limits_{j=1}^{q} (\beta_{j})_n}\cdot\frac{Z^n}{n!} \int_{C}e^{-t}t^{v+n-1} dt \right].
      \end{align*}
       where, $\alpha_i=a_{1i}+i_2a_{2i}=\alpha_{1i}e_1+\alpha_{2i}e_2,\; \beta_j=b_{1j}+i_2b_{2j}=\beta_{1j}e_1+\beta_{2j}e_2,\;t=t_1e_1+t_2e_2 \in{\mathbb{BC}}$ for $i=1,2,\ldots, p$ and $j=1,2,\ldots,q$ with $\beta_{1j}$ and $\beta_{2j}$ are neither zero nor a negative integer and $C(t)$ be a curve in $\mathbb{BC}$ whose parametric representation is $C(t)=(C_1(t_1),C_2(t)$ with $0\leq t_1,t_2\leq \infty$.
  \end{theo}
  \begin{proof}
      Using Theorem \ref{th:c1} and \eqref{eq:ab}, we have
      \begin{align*}
          &{}_{p+1} F_{q}\left[\begin{matrix}&v, &a_1,\ldots, &a_{p};\\&b_1,&b_2,\ldots,&b_
         {q};&\end{matrix}Z\right]\\&={}_{p+1} F_{q}\left[\begin{matrix}&v_{11}, &a_{11},\ldots, &a_{1p};\\&b_{11},&b_{12},\ldots,&b_
         {1q};&\end{matrix}z_1\right]e_1+{}_{p+1} F_{q}\left[\begin{matrix}&v_{12}, &a_{11},\ldots, &a_{1p};\\&b_{11},&b_{12},\ldots,&b_
         {1q};&\end{matrix}z_2\right]e_2\\
         &=\frac{1}{\Gamma(v_{11})} \left[ \int_{0}^{\infty}e^{-t_1}t_1^{v_{11}-1}{}_{p} F_{q}\left[\begin{matrix}&a_{11},\ldots, &a_{1p};\\&b_{11},\ldots,&b_{1q};&\end{matrix}(z_1t_1)\right] dt_1 \right]e_1\\
         &\quad+\frac{1}{\Gamma(v_{12})} \left[ \int_{0}^{\infty}e^{-t_2}t_2^{v_{12}-1}{}_{p} F_{q}\left[\begin{matrix}&a_{22},\ldots, &a_{2p};\\&b_{21},\ldots,&b_{2q};&\end{matrix}(z_2t_2)\right] dt_2 \right]e_2\\
         &=\frac{1}{\Gamma(v_{11})} \left[ \int_{C_1}e^{-t_1}t_1^{v_{11}-1}\sum_{n=0}^{\infty}\frac{\prod\limits_{i=1}^{p} (\alpha_{1i})_n}{\prod\limits_{j=1}^{q} (\beta_{1j})_n}\cdot\frac{(z_1t_1)^n}{n!} dt_1 \right]e_1\\
         &\quad+\frac{1}{\Gamma(v_{12})} \left[ \int_{C_2}e^{-t_2}t_2^{v_{12}-1}\sum_{n=0}^{\infty}\frac{\prod\limits_{i=1}^{p} (\alpha_{2i})_n}{\prod\limits_{j=1}^{q} (\beta_{2j})_n}\cdot\frac{(z_2t_2)^n}{n!} dt_2\right]e_2\\
         &=\frac{1}{\Gamma(v)} \left[ \int_{C}e^{-t}t^{v-1}\sum_{n=0}^{\infty}\frac{\prod\limits_{i=1}^{p} (\alpha_{i})_n}{\prod\limits_{j=1}^{q} (\beta_{j})_n}\cdot\frac{(Zt)^n}{n!} dt \right]\\
         &=\frac{1}{\Gamma(v)}\left[ \sum_{n=0}^{\infty}\frac{\prod\limits_{i=1}^{p} (\alpha_{i})_n}{\prod\limits_{j=1}^{q} (\beta_{j})_n}\cdot\frac{Z^n}{n!} \int_{C}e^{-t}t^{v+n-1} dt \right].
      \end{align*}
       Hence, the proof is completed.
  \end{proof}

  \begin{corollary}
     Setting $p=q=0$ in Theorem \ref{th:mn}, we obtain the integral representation of ${}_1F_{0}(v;.;Z)$ as follows
      \begin{align*}
           {}_1F_{0}(v;.;Z)=\frac{1}{\Gamma(v)}\int_{C}e^{(Z-1)t}t^{v-1}dt,
      \end{align*}
        where $Z\in\mathbb{B}_h(0,1)$, $v=v_1+j v_2,\; t=t_1e_1+t_2e_2\in\mathbb{BC}$  with satisfies the inequality $\Re(v_1)>|\Im(v_2)|$ and  $C(t)=(C_1(t_1),C_2(t_2))$ be a curve in $\mathbb{BC}$ with $0\leq t_1,t_2\leq \infty$.
  \end{corollary}
 \begin{example}
     The integral representation of $\frac{1}{(1-Z)^3}$ is given by
     \begin{align*}
        \frac{1}{(1-Z)^3}={}_1F_{0}(3;.;Z)=\frac{1}{\Gamma(3)}\int_{C}t^2e^{(Z-1)t}dt,
     \end{align*}
      where $Z\in\mathbb{B}_h(0,1),\; t=t_1e_1+t_2e_2\in\mathbb{BC}$ and $C(t)=(C_1(t_1),C_2(t_2))$ be a curve in $\mathbb{BC}$ with $0\leq t_1,t_2\leq \infty$.
 \end{example}
  \begin{theo}\label{th:8t}
    If $Z\in\mathbb{B}_h(0,1)$, $m=m_{11}e_1+m_{21}e_2,\;n=n_{11}e_1+n_{21},\;u=u_{11}e_1+u_{21}e_2,\;v=v_{11}e_1+v_{21}e_2\in\mathbb{BC}$ with  $\Re(m_{11})>0,\Re(n_{11})>0$ then integral representation of bicomplex generalized hypergeometric is given by
    \begin{align*}
        &\int_C\int_Du^{m-1}v^{n-1}(1-u)^{n}{}_{p} F_{q}\left[\begin{matrix}&\alpha_1, &\alpha_2,\ldots, &\alpha_{p};\\&\beta_1,&\beta_2,\ldots,.&\beta_
         {q};&\end{matrix}(1-u)(1-v)Z\right]dudv\\&=\frac{\Gamma(m)\Gamma(n)}{\Gamma(m+n+1)}{}_{p+1} F_{q+1}\left[\begin{matrix}&\alpha_1,\ldots, &\alpha_{p},&1;\\&\beta_1,\ldots,&\beta_
         {q},&m+n+1;&\end{matrix}Z\right],
    \end{align*}
    where $C=(C_1,C_2)$ and $D=(D_1,D_2)$ are two curves in $\mathbb{BC}$ whose parametric equations are  $C=\left(C(u_{11}),C(u_{21})\right)$ and  $D=\left(D(v_{11}),D(v_{21})\right)$  respectively with $0\leq u_{11},u_{21},v_{11},v_{21}\leq1$.
  \end{theo}
  \begin{proof}
      Let $k$ be a non negative integer,
      \begin{align}\label{eq:2g}
         & \frac{\Gamma(m)\Gamma(n)(1)_k}{\Gamma(m+n+1)(m+n+1)_k}\nonumber\\&=\frac{\Gamma(m_{11})\Gamma(n_{11})(1)_k}{\Gamma(m_{11}+n_{11}+1)(m_{11}+n_{11}+1)_k}e_1+\frac{\Gamma(m_{21})\Gamma(n_{21})(1)_k}{\Gamma(m_{21}+n_{21}+1)(m_{21}+n_{21}+1)_k}e_2\nonumber\\
          &=\frac{\Gamma(m_{11})\Gamma(n_{11})\Gamma(k+1)}{\Gamma(m_{11}+n_{11}+k+1)}e_1+\frac{\Gamma(m_{21})\Gamma(n_{21})\Gamma(k+1)}{\Gamma(m_{21}+n_{21}+k+1)}e_2\nonumber\\
          &=\frac{\Gamma(m_{11})\Gamma(n_{11}+k+1)}{\Gamma(m_{11}+n_{11}+k+1)}\times \frac{\Gamma(n_{11})\Gamma(k+1)}{\Gamma(n_{11}+k+1)}e_1+\frac{\Gamma(m_{21})\Gamma(n_{21}+k+1)}{\Gamma(m_{21}+n_{21}+k+1)}\times \frac{\Gamma(n_{21})\Gamma(k+1)}{\Gamma(n_{21}+k+1)}e_2\nonumber\\
          &=B(m_{11},m_{11}+n_{11}+k+1)\times B(n_{11},k+1)e_1+B(m_{21},m_{21}+n_{21}+k+1)\times B(n_{21},k+1)e_2\nonumber\\
          &=\int_0^1u_{11}^{m_{11}-1}(1-u_{11})^{n_{11}+k}du_{11}\times\int_0^1v_{11}^{n_{11}-1}(1-v_{11})^kdv_{11}e_1\nonumber\\
          &\quad+\int_0^1u_{21}^{m_{21}-1}(1-u_{21})^{n_{21}+k}du_{21}\times\int_0^1v_{21}^{n_{21}-1}(1-v_{21})^kdv_{21}e_2.
        \end{align}
    Let us consider two curves $C=(C_1,C_2)$ and $D=(D_1,D_2)$ in $\mathbb{BC}$, whose parametric equations are $C=\left(C(u_{11}),C(u_{21})\right)$ and  $D=\left(D(v_{11}),D(v_{21})\right)$ respectively, where $0\leq u_{11},u_{21},v_{11},v_{21}\leq1$. Then from the equation \eqref{eq:2g}, we get
        \begin{align}\label{eq:9f}
          & \frac{\Gamma(m)\Gamma(n)(1)_k}{\Gamma(m+n+1)(m+n+1)_k}\nonumber\\ &=\int_C\int_D[(u_{11}e_1+u_{21}e_2)^{m_{11}e_1+m_{21}e_2-1}(v_{11}e_1+v_{21}e_2)^{n_{11}e_1+n_{21}e_2-1}\nonumber\\&\hspace{30pt}(1-u_{11}e_1-u_{21}e_2)^{n_{11}e_1+n_{21}e_2+k}(1-v_{11}e_1-v_{21}e_2)^k](du_{11}e_1+du_{21}e_2)(dv_{11}e_1+dv_{21}e_2)\nonumber\\
          &=\int_C\int_Du^{m-1}v^{n-1}(1-u)^{n+k}(1-v)^kdudv.
      \end{align}
      Now using \eqref{eq:9f}, we have
      \begin{align*}
          &\int_C\int_Du^{m-1}v^{n-1}(1-u)^{n}{}_{p} F_{q}\left[\begin{matrix}&\alpha_1, &\alpha_2,\ldots, &\alpha_{p};\\&\beta_1,&\beta_2,\ldots,&\beta_
         {q};&\end{matrix}(1-u)(1-v)Z\right]dudv\\
         &=\int_C\int_Du^{m-1}v^{n-1}(1-u)^{n}\sum_{k=0}^{\infty}\frac{\prod\limits_{i=1}^{p}(\alpha_i)_k}{\prod\limits_{j=1}^{q}(\beta_j)_k}.\frac{{(1-u)^k(1-v)^kZ}^k}{k!}dudv\\
         &=\sum_{k=0}^{\infty}\frac{\prod\limits_{i=1}^{p}(\alpha_i)_k}{\prod\limits_{j=1}^{q}(\beta_j)_k}.\frac{Z^k}{k!}\int_C\int_Du^{m-1}v^{n-1}(1-u)^{n+k}(1-v)^kdudv\\
         &=\sum_{k=0}^{\infty}\frac{\prod\limits_{i=1}^{p}(\alpha_i)_k}{\prod\limits_{j=1}^{q}(\beta_j)_k}.\frac{Z^k}{k!}\frac{\Gamma(m)\Gamma(n)(1)_k}{\Gamma(m+n+1)(m+n+1)_k}\\
         &=\frac{\Gamma(m)\Gamma(n)}{\Gamma(m+n+1)}{}_{p+1} F_{q+1}\left[\begin{matrix}&\alpha_1,\ldots, &\alpha_{p},&1;\\&\beta_1,\ldots,&\beta_
         {q},&m+n+1;&\end{matrix}Z\right],
      \end{align*}
      which completes the proof of the theorem.
  \end{proof}
  \begin{corollary}
      Setting $p=q=1$ and $\alpha_1=\beta_1$ in Theorem \ref{th:8t}, we get
       \begin{align*}
        &\int_C\int_Du^{m-1}v^{n-1}(1-u)^{n}e^{(1-u)(1-v)Z}dudv=\frac{\Gamma(m)\Gamma(n)}{\Gamma(m+n+1)}{}_{2} F_{2}\left[\begin{matrix}&\alpha_1,&1;\\&\alpha_1,&m+n+1;&\end{matrix}Z\right],
    \end{align*}
    where $Z\in\mathbb{B}_h(0,1)$, $m=m_{11}e_1+m_{21}e_2,n=n_{11}e_1+n_{21},u=u_{11}e_1+u_{21}e_2,v=v_{11}e_1+v_{21}e_2\in\mathbb{BC}$ with satisfies the condition $\Re(m_{11})>0,\Re(n_{11})>0$ and $C=\left(C(u_{11}),C(u_{21})\right)$, $D=\left(D(v_{11}),D(v_{21})\right)$ are two curves in $\mathbb{BC}$ with $0\leq u_{11},u_{21},v_{11},v_{21}\leq1$.
  \end{corollary}
\section{Some quadratic transformations}
 In this section, we have constructed some quadratic transforms of the bicomplex generalized hypergeometric function.
 \begin{theo}\label{th}
  If  $Z=z_1 e_1+z_2 e_2\in\mathbb{B}_h(0,1)$, $\alpha_i=\alpha_{1i}e_1+\alpha_{2i}e_2,\;\beta_j=\beta_{1j}e_1+\beta_{2j}e_2 \in{\mathbb{BC}}$ for $i=1,2,\ldots,p$ and $j=1,2,\ldots,q$ with $\beta_{1j}$ and $\beta_{2j}$ are neither zero nor a negative integer, then
 \begin{align*}
   2 \times{}_{2p} F_{2q+1}\left[\begin{matrix}&\frac{\alpha_1}{2},&\frac{\alpha_2}{2},\ldots,&\frac{\alpha_{p}}{2},\frac{\alpha_1+1}{2}, &\frac{\alpha_2+1}{2},\ldots, &\frac{\alpha_{p}+1}{2};\\&\frac{1}{2},&\frac{\beta_1}{2},\ldots,&\frac{\beta_{q}}{2},\frac{\beta_1+1}{2}, &\frac{\beta_2+1}{2},\ldots, &\frac{\beta_{q}+1}{2};&\end{matrix}\frac{Z^2}{(4)^{q+1-p}}\right]&={}_{p} F_{q}\left[\begin{matrix}&\alpha_1, &\alpha_2,\ldots, &\alpha_{p_1};\\&\beta_1,&\beta_2,\ldots,&\beta_{p_2};&\end{matrix}Z\right]\\&\quad+{}_{p} F_{q}\left[\begin{matrix}&\alpha_1, &\alpha_2,\ldots, &\alpha_{p};\\&\beta_1,&\beta_2,\ldots,&\beta_{q};&\end{matrix}-Z\right].
  \end{align*}
\end{theo}
\begin{proof}\label{eq:10}
For integral $k\geq1$, we have
\begin{align}\label{eq:sn}
    (\alpha_{1i})_{2k}&=(\alpha_{1i})(\alpha_{1i}+1)\ldots(\alpha_{1i}+2k-1)\nonumber\\
    &=(\alpha_{1i})(\alpha_{1i}+2)\ldots(\alpha_{1i}+2k-2)\times(\alpha_{1i}+1)(\alpha_{1i}+3)\ldots(\alpha_{1i}+2k-1)\nonumber\\
    &=2^{2k}\left(\frac{\alpha_{1i}}{2}\right)\left(\frac{\alpha_{1i}}{2}+1\right)\ldots\left(\frac{\alpha_{1i}}{2}+k-1\right)\nonumber\\&\quad\times(\frac{\alpha_{1i}+1}{2})\left(\frac{\alpha_{1i}+1}{2}+1
    \right)\ldots\left(\frac{\alpha_{1i}+1}{2}+k-1\right)\nonumber\\
    &=2^{2k}\left(\frac{\alpha_{1i}}{2}\right)_k\left(\frac{\alpha_{1i}+1}{2}\right)_k.
\end{align}
Similarly, we get
\begin{align}\label{eq:11}
    (\beta_{1j})_{2k}=2^{2k}\left(\frac{\beta_{1j}}{2}\right)_k\left(\frac{\beta_{1j}+1}{2}\right)_k.
\end{align}
Again for $k\in \mathbb{N}$, we have
\begin{align}\label{eq:12}
    (2k)!=2^{2k}k!\left(\frac{1}{2}\right)_k.
\end{align}
Using the series form \eqref{eq:gh} of the bicomplex generalized hypergeometric function, we obtain
\begin{align}\label{eq:13}
   & {}_{2p} F_{2q+1}\left[\begin{matrix}&\frac{\alpha_1}{2},&\frac{\alpha_2}{2} ,\ldots,&\frac{\alpha_{p}}{2},\frac{\alpha_1+1}{2}, &\frac{\alpha_2+1}{2},\ldots, &\frac{\alpha_{p}+1}{2};\\&\frac{1}{2}&\frac{\beta_1}{2},\ldots,&\frac{\beta_{q}}{2},\frac{\beta_1+1}{2}, &\frac{\beta_2+1}{2},\ldots, &\frac{\beta_{q}+1}{2};&\end{matrix}\frac{Z^2}{(4)^{q+1-p}}\right] \nonumber\\
   &=\sum_{k=0}^{\infty}\frac{\prod\limits_{i=1}^{p}\left(\frac{\alpha_i}{2}\right)_k\left(\frac{\alpha_i+1}{2}\right)_k}{\left(\frac{1}{2}\right)_{k}\prod\limits_{j=1}^{q}(\frac{\beta_j}{2})_k(\frac{\beta_j+1}{2})_k}\cdot\frac{\left(\frac{Z^2}{4^{q+1-p}}\right)^k}{k!} \nonumber\\
 &=\sum_{k=0}^{\infty}\frac{\prod\limits_{i=1}^{p}\left[2^{2k}\left(\frac{\alpha_{1i}}{2}\right)_k\left(\frac{\alpha_{1i}+1}{2}\right)_k\right]}{\prod\limits_{j=1}^{q}\left[2^{2k}(\frac{\beta_{1j}}{2})_k(\frac{\beta_{1j}+1}{2})_k\right]}\cdot\frac{\left(z_1^2\right)^k}{2^{2k}\left(\frac{1}{2}\right)_kk!}e_1+\sum_{k=0}^{\infty}\frac{\prod\limits_{i=1}^{p}\left[2^{2k}\left(\frac{\alpha_{2i}}{2}\right)_k\left(\frac{\alpha_{2i}+1}{2}\right)_k\right]}{\prod\limits_{j=1}^{q}\left[2^{2k}(\frac{\beta_{2j}}{2})_k(\frac{\beta_{2j}+1}{2})_k\right]}\cdot\frac{\left(z_2^2\right)^k}{2^{2k}\left(\frac{1}{2}\right)_kk!}e_2.
\end{align}
Using \eqref{eq:sn}, \eqref{eq:11}, \eqref{eq:12} and \eqref{eq:13}, we get
\begin{align*}
    &{}_{2p} F_{2q+1}\left[\begin{matrix}&\frac{\alpha_1}{2},&\frac{\alpha_2}{2} ,\ldots,&\frac{\alpha_{p}}{2},\frac{\alpha_1+1}{2}, &\frac{\alpha_2+1}{2},\ldots, &\frac{\alpha_{p}+1}{2};\\&\frac{1}{2}&\frac{\beta_1}{2},\ldots,&\frac{\beta_{q}}{2},\frac{\beta_1+1}{2}, &\frac{\beta_2+1}{2},\ldots, &\frac{\beta_{q}+1}{2};&\end{matrix}\frac{Z^2}{(4)^{p_2+1-P_{1}}}\right]\\
&=\left[\sum_{k=0}^{\infty}\frac{\prod\limits_{i=1}^{p}\left[(\alpha_{1i})_{2k}\right]}{\prod\limits_{j=1}^{q}\left[(\beta_{1j})_{2k}\right]}\cdot\frac{z_1^{2k}}{2k!}\right]e_1+\left[\sum_{k=0}^{\infty}\frac{\prod\limits_{i=1}^{p}\left[(\alpha_{2i})_{2k}\right]}{\prod\limits_{j=1}^{q}\left[(\beta_{2j})_{2k}\right]}\cdot\frac{z_2^{2k}}{2k!}\right]e_2\\
    &=\frac{1}{2}\left[\sum_{n=0}^{\infty}\frac{\prod\limits_{i=1}^{p}\left[(\alpha_{1i})_{n}\right]}{\prod\limits_{j=1}^{q}\left[(\beta_{1j})_{n}\right]}\cdot\frac{(z_1)^{n}+(-z_1)^n}{n!}\right]e_1+\frac{1}{2}\left[\sum_{n=0}^{\infty}\frac{\prod\limits_{i=1}^{p}\left[(\alpha_{2i})_{n}\right]}{\prod\limits_{j=1}^{q}\left[(\beta_{2j})_{n}\right]}\cdot\frac{(z_2)^{n}+(-z_2)^n}{n!}\right]e_2\\
    &=\frac{1}{2}\left[\sum_{n=0}^{\infty}\frac{\prod\limits_{i=1}^{p}\left[(\alpha_{1i})_{n}\right]}{\prod\limits_{j=1}^{q}\left[(\beta_{1j})_{n}\right]}\cdot\frac{(z_1)^{n}}{n!}e_1+\sum_{n=0}^{\infty}\frac{\prod\limits_{i=1}^{p}\left[(\alpha_{2i})_{n}\right]}{\prod\limits_{j=1}^{q}\left[(\beta_{2j})_{n}\right]}\cdot\frac{(z_2)^{n}}{n!}e_2\right]\\
    &\quad+\frac{1}{2}\left[\sum_{n=0}^{\infty}\frac{\prod\limits_{i=1}^{p}\left[(\alpha_{1i})_{n}\right]}{\prod\limits_{j=1}^{q}\left[(\beta_{1j})_{n}\right]}\cdot\frac{(-z_1)^{n}}{n!}e_1+\sum_{n=0}^{\infty}\frac{\prod\limits_{i=1}^{p}\left[(\alpha_{2i})_{n}\right]}{\prod\limits_{j=1}^{q}\left[(\beta_{2j})_{n}\right]}\cdot\frac{(-z_2)^{n}}{n!}e_2\right]\\
    &=\frac{1}{2}\left[\sum_{n=0}^{\infty}\frac{\prod\limits_{i=1}^{p}\left[(\alpha_i)_{n}\right]}{\prod\limits_{j=1}^{q}\left[(\beta_j)_{n}\right]}\cdot\frac{Z^n}{n!}\right]+\frac{1}{2}\left[\sum_{n=0}^{\infty}\frac{\prod\limits_{i=1}^{p}\left[(\alpha_i)_{n}\right]}{\prod\limits_{j=1}^{q}\left[(\beta_j)_{n}\right]}\cdot\frac{(-Z)^n}{n!}\right]\\
    &=\frac{1}{2}\cdot{}_{p} F_{q}\left[\begin{matrix}&\alpha_1, &\alpha_2,\ldots, &\alpha_{p};\\&\beta_1,&\beta_2,\ldots,&\beta_{q};&\end{matrix}Z\right]+ \frac{1}{2}\cdot{}_{p} F_{q}\left[\begin{matrix}&\alpha_1, &\alpha_2,\ldots, &\alpha_{p};\\&\beta_1,&\beta_2,\ldots,&\beta_{q};&\end{matrix}-Z\right],
\end{align*}
 which completes the proof of the theorem.
\end{proof}
\begin{theo}\label{co}
   If  $Z=z_1 e_1+z_2 e_2\in\mathbb{B}_h(0,1)$, $\alpha_i=\alpha_{1i}e_1+\alpha_{2i}e_2,\; \beta_j=\beta_{1j}e_1+\beta_{2j}e_2 \in{\mathbb{BC}}$ for $i=1,2,\ldots,p$ and $j=1,2,\ldots,q$ with $\beta_{1j}$ and $\beta_{2j}$ are neither zero nor a negative integer, then
\begin{align*}
     &\frac{2Z\prod\limits_{i=1}^{p}(\alpha_i)}{\prod\limits_{j=1}^{q}(\beta_j)} \times {}_{2p} F_{2q+1}\left[\begin{matrix}&\frac{\alpha_1+1}{2},&\frac{\alpha_2+1}{2} ,\ldots,&\frac{\alpha_{p}+1}{2},\frac{\alpha_1+2}{2}, &\frac{\alpha_2+2}{2},\ldots, &\frac{\alpha_{p}+2}{2};\\&\frac{3}{2},&\frac{\beta_1+1}{2},\ldots,&\frac{\beta_{q}+1}{2},\frac{\beta_1+2}{2}, &\frac{\beta_2+2}{2},\ldots, &\frac{\beta_{q}+2}{2};&\end{matrix}\frac{Z^2}{(4)^{q+1-p}}\right]\\&={}_{p} F_{q}\left[\begin{matrix}&\alpha_1, &\alpha_2,\ldots, &\alpha_{p};\\&\beta_1,&\beta_2,\ldots,&\beta_{q};&\end{matrix}Z\right]-{}_{p} F_{q}\left[\begin{matrix}&\alpha_1, &\alpha_2,\ldots, &\alpha_{p};\\&\beta_1,&\beta_2,\ldots,&\beta_{q};&\end{matrix}-Z\right].
\end{align*}
\end{theo}
\begin{proof}
   Let $k$ be a non-negative integer. Then
   \begin{align*}
       (\alpha_{1i})_{2k+1}&=(\alpha_{1i})(\alpha_{1i}+1)\ldots(\alpha_{1i}+2k)\nonumber\\
       &=2^{2k}(\alpha_{1i})\left[\left(\frac{\alpha_{1i}+2}{2}\right)\left(\frac{\alpha_{1i}+2}{2}+1\right)\ldots\left(\frac{\alpha_{1i}+2}{2}+k-1\right)\right]\nonumber\\
       & \quad\times\left[\left(\frac{\alpha_{1i}+1}{2}\right)\left(\frac{\alpha_{1i}+1}{2}+1\right)\ldots\left(\frac{\alpha_{1i}+1}{2}+k-1\right)\right]\nonumber\\
       &=2^{2k}(\alpha_{1i})\left(\frac{\alpha_{1i}+2}{2}\right)_k\left(\frac{\alpha_{1i}+1}{2}\right)_k
       \end{align*}
        Similarly, we get
       \begin{align*}
            (\alpha_{2i})_{2k+1}&=2^{2k}(\alpha_{2i})\left(\frac{\alpha_{2i}+2}{2}\right)_k\left(\frac{\alpha_{2i}+1}{2}\right)_k.
            \end{align*}
       Now, using idempotent representation of $(\alpha_i)_{2k+1}$, we have
       \begin{align}\label{eq:4r}
           (\alpha_i)_{2k+1}&=(\alpha_{1i})_{2k+1}e_1+(\alpha_{2i})_{2k+1}e_2\nonumber\\
           &=2^{2k}(\alpha_{1i})\left(\frac{\alpha_{1i}+2}{2}\right)_k\left(\frac{\alpha_{1i}+1}{2}\right)_ke_1+2^{2k}(\alpha_{1i})\left(\frac{\alpha_{1i}+2}{2}\right)_k\left(\frac{\alpha_{1i}+1}{2}\right)_ke_2\nonumber\\
           &=2^{2k}(\alpha_{i})\left(\frac{\alpha_{i}+2}{2}\right)_k\left(\frac{\alpha_{i}+1}{2}\right)_k.
       \end{align}
       Similarly, we obtain
       \begin{align}\label{eq:5t}
           (\beta_j)_{2k+1}=2^{2k}(\beta_{j})\left(\frac{\beta_{j}+2}{2}\right)_k\left(\frac{\beta_{j}+1}{2}\right)_k.
       \end{align}
       Again for $k\in\mathbb{N}$, we have
       \begin{align}
           (2k+1)!=2^{2k}k!\left(\frac{3}{2}\right)_k.
       \end{align}
       \begin{align*}
           &\frac{2Z\prod\limits_{i=1}^{p}(\alpha_i)}{\prod\limits_{j=1}^{q}(\beta_j)} \times {}_{2p} F_{2q+1}\left[\begin{matrix}&\frac{\alpha_1+1}{2},&\frac{\alpha_2+1}{2} ,\ldots,&\frac{\alpha_{p}+1}{2},\frac{\alpha_1+2}{2}, &\frac{\alpha_2+2}{2},\ldots, &\frac{\alpha_{p}+2}{2};\\&\frac{3}{2},&\frac{\beta_1+1}{2}, \ldots,&\frac{\beta_{q}+1}{2},\frac{\beta_1+2}{2}, &\frac{\beta_2+2}{2},\ldots, &\frac{\beta_{q}+2}{2};&\end{matrix}\frac{Z^2}{(4)^{q+1-p}}\right]\\
           &=2Z\sum_{k=0}^\infty\frac{\prod\limits_{i=1}^{p}(\alpha_i)\left(\frac{\alpha_i+1}{2}\right)_k\left(\frac{\alpha_i+2}{2}\right)_k}{(\frac{3}{2})_k\prod\limits_{j=1}^{q}(\beta_j)\left(\frac{\beta_i+1}{2}\right)_k\left(\frac{\beta_i+2}{2}\right)_k}\frac{Z^{2k}}{k!(4)^{kq+k-kp}}\\
           &=2Z\sum_{k=0}^\infty\frac{\prod\limits_{i=1}^{p}\left[2^{2k}(\alpha_i)\left(\frac{\alpha_i+1}{2}\right)_k\left(\frac{\alpha_i+2}{2}\right)_k\right]}{\prod\limits_{j=1}^{q}\left[2^{2k}(\beta_j)\left(\frac{\beta_i+1}{2}\right)_k\left(\frac{\beta_i+2}{2}\right)_k\right]}\frac{Z^{2k}}{2^{2k}k!(\frac{3}{2})_k}\\
           &=2\sum_{k=0}^\infty\frac{\prod\limits_{i=1}^{p}(\alpha_i)_{2k+1}}{\prod\limits_{j=1}^{q}(\beta_j)_{2k+1}}\frac{Z^{2k+1}}{(2k+1)!}\\
           &=\sum_{k=0}^\infty\frac{\prod\limits_{i=1}^{p}(\alpha_i)_k}{\prod\limits_{j=1}^{q}(\beta_j)_k}\left[\frac{Z^k-(-Z)^k}{k!}\right]\\
           &=\sum_{k=0}^\infty\frac{\prod\limits_{i=1}^{p}(\alpha_i)_k}{\prod\limits_{j=1}^{q}(\beta_j)_k}\left[\frac{Z^k}{k!}\right]-\sum_{k=0}^\infty\frac{\prod\limits_{i=1}^{p}(\alpha_i)_k}{\prod\limits_{j=1}^{q}(\beta_j)_k}\left[\frac{(-Z)^k}{k!}\right]\\
           &={}_{p} F_{q}\left[\begin{matrix}&\alpha_1, &\alpha_2,\ldots, &\alpha_{p};\\&\beta_1,&\beta_2,\ldots,&\beta_{q};&\end{matrix}Z\right]-{}_{p} F_{q}\left[\begin{matrix}&\alpha_1, &\alpha_2,\ldots, &\alpha_{p};\\&\beta_1,&\beta_2,\ldots,&\beta_{q};&\end{matrix}-Z\right].
       \end{align*}
       Hence the proof is completed.
 \end{proof}

\begin{theo}
    Let $n$ be a non-negative integer and $\alpha_1=\alpha_{11}e_1+\alpha_{21}e_2,\;\alpha_2=\alpha_{12}e_1+\alpha_{22}e_2$ and $\beta_{1}=\beta_{11}e_1+\beta_{21}e_2 \in \mathbb{BC}$ are all independent of $n$, then
    \begin{align*}
        {}_{3} F_{2}\left[\begin{matrix}& &\quad\quad-n,\alpha_1, \alpha_2;\\&\beta_1,& 1-\beta_1+\alpha_1+\alpha_2-n;&\end{matrix}1\right]=\frac{(\beta_1-\alpha_1)_n(\beta_1-\alpha_2)_n}{(\beta_1)_n(\beta_1-\alpha_1-\alpha_2)_n}.
    \end{align*}
\end{theo}
\begin{proof}
    Using idempotent representation \eqref{eq:ab} and Theorem\ref{th:s}, we obtain
    \begin{align*}
        &{}_{3} F_{2}\left[\begin{matrix}& &\quad\quad-n,\alpha_1, \alpha_2;\\&\beta_1,& 1-\beta_1+\alpha_1+\alpha_2-n;&\end{matrix}1\right]\\&={}_{3} F_{2}\left[\begin{matrix}& &\quad\quad-n,\alpha_{11}, \alpha_{12};\\&\beta_{11},& 1-\beta_{11}+\alpha_{11}+\alpha_{12}-n;&\end{matrix}1\right]e_1+{}_{3} F_{2}\left[\begin{matrix}& &\quad\quad-n,\alpha_{21}, \alpha_{22};\\&\beta_{21},& 1-\beta_{21}+\alpha_{21}+\alpha_{22}-n;&\end{matrix}1\right]e_2\\
        &=\frac{(\beta_{11}-\alpha_{11})_n(\beta_{11}-\alpha_{12})_n}{(\beta_{11})_n(\beta_{11}-\alpha_{11}-\alpha_{12})_n}e_1+\frac{(\beta_{21}-\alpha_{21})_n(\beta_{21}-\alpha_{22})_n}{(\beta_{21})_n(\beta_{21}-\alpha_{21}-\alpha_{22})_n}e_2\\
        &=\frac{(\beta_1-\alpha_1)_n(\beta_1-\alpha_2)_n}{(\beta_1)_n(\beta_1-\alpha_1-\alpha_2)_n},
    \end{align*}
which completes the proof of the theorem.
\end{proof}
\section{Differential relation and analyticity of the bicomplex generalized hypergeometric function}
In this section, we derive differential relationship and check analyticity of the bicomplex generalized hypergeometric function.
\begin{theo}
  Let  $Z=z_1 e_1+z_2 e_2\in\mathbb{B}_h(0,1)$, $\alpha_i=\alpha_{1i}e_1+\alpha_{2i}e_2,\; \beta_j=\beta_{1j}e_1+\beta_{2j}e_2 \in{\mathbb{BC}}$ for $i=1,2,\ldots,p$ and $j=1,2,\ldots,q$ with $\beta_{1j}$ and $\beta_{2j}$ are neither zero nor a negative integer, then
\begin{align*}
    \frac{d^k}{dZ^k}\left({}_{p} F_{q}\left[\begin{matrix}&\alpha_1, &\alpha_2,\ldots, &\alpha_{p};\\&\beta_1,&\beta_2,\ldots,&\beta_{q};&\end{matrix}Z\right]\right)=\frac{\prod\limits_{i=1}^{p}(\alpha_i)_k}{\prod\limits_{j=1}^{q}(\beta_j)_k}{}_{p} F_{q}\left[\begin{matrix}&\alpha_1+k, &\alpha_2+k,\ldots, &\alpha_{p}+k;\\&\beta_1+k,&\beta_2+k,\ldots,&\beta_{q}+k;&\end{matrix}Z\right].
\end{align*}
\end{theo}
\begin{proof}
    Using power series representation \eqref{eq:gh} of the generalized bicomplex hypergeometric function, we have
\begin{align*}
    & \frac{d}{dZ} \left( {}_{p} F_{q}\left[\begin{matrix}&\alpha_1, &\alpha_2,.... &\alpha_{p};\\&\beta_1,&\beta_2,.....&\beta_{q};&\end{matrix}Z\right]\right)\\
     &=\frac{d}{dz_1}\left(\sum_{n=0}^{\infty} \frac{\prod\limits_{i=1}^{p} (\alpha_{1i})_n}{\prod\limits_{j=1}^{q} (\beta_{1j})_n}.\frac{z_1^n}{n!}\right)e_1+\frac{d}{dz_2}\left(\sum_{n=0}^{\infty} \frac{\prod\limits_{i=1}^{p}(\alpha_{2i})_n}{\prod\limits_{j=1}^{q}(\beta_{2j})_n}.\frac{z_2^n}{n!}\right) e_2\nonumber\\
     &=\left(\sum_{n=1}^{\infty}\frac{\prod\limits_{i=1}^{p}(\alpha_{1i})_n}{\prod\limits_{j=1}^{q}(\beta_{1j})_n}.\frac{nz_1^{n-1}}{n!} \right)e_1+\left(\sum_{n=1}^{\infty}\frac{\prod\limits_{i=1}^{p}(\alpha_{2i})_n}{\prod\limits_{j=1}^{q}(\beta_{2j})_n}.\frac{nz_2^{n-1}}{n!}\right) e_2\nonumber\\
     &=\left(\sum_{n=0}^{\infty}\frac{\prod\limits_{i=1}^{p}(\alpha_{1i})_{n+1}}{\prod\limits_{j=1}^{q}(\beta_{1j})_{n+1}}.\frac{z_1^{n}}{n!} \right)e_1+\left(\sum_{n=0}^{\infty}\frac{\prod\limits_{i=1}^{p}(\alpha_{2i})_{n+1}}{\prod\limits_{j=1}^{q}(\beta_{2j})_{n+1}}.\frac{z_2^{n}}{n!}\right) e_2\\
    &=\frac{\prod\limits_{i=1}^{p}(\alpha_{1i})}{\prod\limits_{j=1}^{q}(\beta_{1j})}{}\left(\sum_{n=0}^{\infty}\frac{\prod\limits_{i=1}^{p}(\alpha_{1i}+1)_{n}}{\prod\limits_{j=1}^{q}(\beta_{1j}+1)_{n}}.\frac{z_1^{n}}{n!} \right)e_1+\frac{\prod\limits_{i=1}^{p}(\alpha_{2i})}{\prod\limits_{j=1}^{q}(\beta_{2j})}\left(\sum_{n=0}^{\infty}\frac{\prod\limits_{i=1}^{p}(\alpha_{2i}+1)_{n}}{\prod\limits_{j=1}^{q}(\beta_{2j}+1)_{n}}.\frac{z_2^{n}}{n!}\right)e_2\\
    &=\left(\frac{\prod\limits_{i=1}^{p}(\alpha_{1i})}{\prod\limits_{j=1}^{q}(\beta_{1j})}e_1+\frac{\prod\limits_{i=1}^{p}(\alpha_{2i})}{\prod\limits_{j=1}^{q}(\beta_{2j})}e_2\right)\left(\sum_{n=0}^{\infty}\frac{\prod\limits_{i=1}^{p}(\alpha_{1i}+1)_{n}}{\prod\limits_{j=1}^{q}(\beta_{1j}+1)_{n}}.\frac{z_1^{n}}{n!}e_1+\sum_{n=0}^{\infty}\frac{\prod\limits_{i=1}^{p}(\alpha_{2i}+1)_{n}}{\prod\limits_{j=1}^{q}(\beta_{2j}+1)_{n}}.\frac{z_2^{n}}{n!}e_2\right)\\
    &=\frac{\prod\limits_{i=1}^{p}(\alpha_i)}{\prod\limits_{j=1}^{q}(\beta_j)}{}\left({}_{p} F_{q}\left[\begin{matrix}&\alpha_1+1, &\alpha_2+1,\ldots, &\alpha_{p}+1;\\&\beta_1+1,&\beta_2+1,\ldots,&\beta_{q}+1;&\end{matrix}Z\right]\right).
    \end{align*}
Therefore, the given statement is true for $k=1$. Assume that given statement is true for $k=m$ and it implies that
\begin{align*}
    &\frac{d^m}{dZ^m} \left( {}_{p} F_{q}\left[\begin{matrix}&\alpha_1, &\alpha_2,\ldots, &\alpha_{p};\\&\beta_1,&\beta_2,\ldots,&\beta_{q};&\end{matrix}Z\right]\right)=\frac{\prod\limits_{i=1}^{p}(\alpha_i)_m}{\prod\limits_{j=1}^{q}(\beta_j)_m}{}\left({}_{p} F_{q}\left[\begin{matrix}&\alpha_1+m, &\alpha_2+m,\ldots, &\alpha_{p}+m;\\&\beta_1+m,&\beta_2+m,\ldots,&\beta_{q}+m;&\end{matrix}Z\right]\right).\\
\end{align*}
Now,
\begin{align*}
     &\frac{d^{m+1}}{dZ^{m+1}} \left( {}_{p} F_{q}\left[\begin{matrix}&\alpha_1, &\alpha_2,\ldots, &\alpha_{p};\\&\beta_1,&\beta_2,\ldots,&\beta_{q};&\end{matrix}Z\right]\right)
     =\frac{d}{dZ}\left[ \frac{d^m}{dZ^m} \left( {}_{p} F_{q}\left[\begin{matrix}&\alpha_1, &\alpha_2,\ldots, &\alpha_{p};\\&\beta_1,&\beta_2,\ldots,&\beta_{q};&\end{matrix}Z\right]\right) \right]\\
     &=\frac{\prod\limits_{i=1}^{p}(\alpha_{i})_m}{\prod\limits_{j=1}^{q}(\beta_j)_m}\left[\frac{d}{dz_1}\left(\sum_{n=0}^{\infty} \frac{\prod\limits_{i=1}^{p} (\alpha_{1i}+m)_n}{\prod\limits_{j=1}^{q} (\beta_{1j}+m)_n}.\frac{z_1^n}{n!}\right)e_1+\frac{d}{dz_2}\left(\sum_{n=0}^{\infty} \frac{\prod\limits_{i=1}^{p}(\alpha_{2i}+m)_n}{\prod\limits_{j=1}^{q}(\beta_{2j}+m)_n}.\frac{z_2^n}{n!}\right) e_2\right]\\
     &=\frac{\prod\limits_{i=1}^{p}(\alpha_i)_m}{\prod\limits_{j=1}^{q}(\beta_j)_m}\left[\left(\sum_{n=0}^{\infty} \frac{\prod\limits_{i=1}^{p} (\alpha_{1i}+m)_{n+1}}{\prod\limits_{j=1}^{q} (\beta_{1j}+m)_{n+1}}.\frac{z_1^n}{n!}\right)e_1+\left(\sum_{n=0}^{\infty} \frac{\prod\limits_{i=1}^{p}(\alpha_{2i}+m)_{n+1}}{\prod\limits_{j=1}^{q}(\beta_{2j}+m)_{n+1}}.\frac{z_2^n}{n!}\right) e_2\right]\\
     &=\frac{\prod\limits_{i=1}^{p}(\alpha_i)_m(\alpha_i+m)}{\prod\limits_{j=1}^{q}(\beta_j)_m(\beta_j+m)}\left(\sum_{n=0}^{\infty}\frac{\prod\limits_{i=1}^{p}(\alpha_{1i}+m+1)_{n}}{\prod\limits_{j=1}^{q}(\beta_{1j}+m+1)_{n}}.\frac{z_1^{n}}{n!} e_1+\sum_{n=0}^{\infty}\frac{\prod\limits_{i=1}^{p}(\alpha_{2i}+m+1)_{n}}{\prod\limits_{j=1}^{q}(\beta_{2j}+m+1)_{n}}.\frac{z_2^{n}}{n!}e_2\right)\\
     &=\frac{\prod\limits_{i=1}^{p}(\alpha_i)_{m+1}}{\prod\limits_{j=1}^{q}(\beta_j)_{m+1}}\left({}_{p} F_{q}\left[\begin{matrix}&\alpha_1+m+1, &\alpha_2+m+1,\ldots, &\alpha_{p}+m+1;\\&\beta_1+m+1,&\beta_2+m+1,\ldots,&\beta_{q}+m+1;&\end{matrix}Z\right]\right),
\end{align*}
which shows that the given statement is true for $k=m+1$. Hence, by Principle of Mathematical Induction, the given statement is true for all $k\in{N}$.
\end{proof}
\begin{theo}
    If  $Z=z_1 e_1+z_2 e_2\in\mathbb{B}_h(0,1)$, $\alpha_i=\alpha_{1i}e_1+\alpha_{2i}e_2,$ $ \beta_j=\beta_{1j}e_1+\beta_{2j}e_2 \in{\mathbb{BC}}$ for $i=1,2,\ldots,p$ and $j=1,2,\ldots,q$ with $\beta_{1j}$ and $\beta_{2j}$ are neither zero nor a negative integer the bicomplex generalized hypergeometric function ${}_{p} F_{q}\left[\begin{matrix}&\alpha_1, &\alpha_2,\ldots, &\alpha_{p};\\&\beta_1,&\beta_2,\ldots,&\beta_{q};&\end{matrix}Z\right]$ is $\mathbb{BC}-holomorphic$ in $\alpha_i$ and $\beta_j$ except when  at least one of $\beta_{1j}$ and $\beta_{2j}$ are either zero or negative integer.
\end{theo}
\begin{proof}
Setting $e_1=\frac{1+k}{2}$ and $e_2=\frac{1-k}{2}$ in \eqref{eq:ab}, we get
\begin{align}\label{eq:19}
    &{}_{p} F_{q}\left[\begin{matrix}&\alpha_1, &\alpha_2,\ldots, &\alpha_{p};\\&\beta_1,&\beta_2,\ldots,&\beta_{q};&\end{matrix}Z\right]= {}_{p} F_{q}\left[\begin{matrix}&\alpha_{11}, &\alpha_{12},\ldots, &\alpha_{1p};\\&\beta_{11},&\beta_{12},\ldots,&\beta_{1q};&\end{matrix}z_1 \right]\left(\frac{1+k}{2}\right) +{}_{p} F_{q}\left[\begin{matrix}&\alpha_{21}, &\alpha_{22},\ldots, &\alpha_{2p};\\&\beta_{21},&\beta_{22},\ldots,&\beta_{2q};&\end{matrix}z_2 \right]\left(\frac{1-k}{2}\right)\nonumber\\
    &=\frac{1}{2}\left(  {}_{p} F_{q}\left[\begin{matrix}&\alpha_{11}, &\alpha_{12},\ldots, &\alpha_{1p};\\&\beta_{11},&\beta_{12},\ldots,&\beta_{1q};&\end{matrix}z_1 \right] + {}_{p} F_{q}\left[\begin{matrix}&\alpha_{21}, &\alpha_{22},\ldots, &\alpha_{2p};\\&\beta_{21},&\beta_{22},\ldots,&\beta_{2q};&\end{matrix}z_2 \right]\right)\nonumber\\
    &\quad+i_2\times \frac{i_1}{2}\left( {}_{p} F_{q}\left[\begin{matrix}&\alpha_{11}, &\alpha_{12},\ldots, &\alpha_{1p};\\&\beta_{11},&\beta_{12},\ldots,&\beta_{1q};&\end{matrix}z_1 \right] -{}_{p} F_{q}\left[\begin{matrix}&\alpha_{21}, &\alpha_{22},\ldots, &\alpha_{2p};\\&\beta_{21},&\beta_{22},\ldots,&\beta_{2q};&\end{matrix}z_2 \right]\right).
\end{align}
Let us consider, ${}_{p} F_{q}\left[\begin{matrix}&\alpha_1, &\alpha_2,\ldots, &\alpha_{p};\\&\beta_1,&\beta_2,\ldots,&\beta_{p};&\end{matrix}Z\right]=f_1(z,z')+i_2f_2(z,z')$. Then from the equation \eqref{eq:19}, we obtain
\begin{align*}
    &f_1(z,z')=\frac{1}{2}\left(  {}_{p} F_{q}\left[\begin{matrix}&\alpha_{11}, &\alpha_{12},\ldots, &\alpha_{1p};\\&\beta_{11},&\beta_{12},\ldots,&\beta_{1q};&\end{matrix}z-i_1z' \right] + {}_{p} F_{q}\left[\begin{matrix}&\alpha_{21}, &\alpha_{22},\ldots, &\alpha_{2p};\\&\beta_{21},&\beta_{22},\ldots,&\beta_{2q};&\end{matrix}z+i_1z' \right]\right) \quad\mbox{and}\quad\\
   &f_2(z,z')= \frac{i_1}{2}( \left(  {}_{p} F_{q}\left[\begin{matrix}&\alpha_{11}, &\alpha_{12},\ldots, &\alpha_{1p};\\&\beta_{11},&\beta_{12},\ldots,&\beta_{1q};&\end{matrix}z-i_1z' \right] - {}_{p} F_{q}\left[\begin{matrix}&\alpha_{21}, &\alpha_{22},\ldots, &\alpha_{2p};\\&\beta_{21},&\beta_{22},\ldots,&\beta_{2q};&\end{matrix}z+i_1z' \right]\right).
\end{align*}
Since, $Z\in\mathbb{B}(0,1)$ which implies $|z_1|<1$ and $|z_2|<1$, then generalized hypergeometric function ${}_{p} F_{q}\left[\begin{matrix}&\alpha_{11}, &\alpha_{12},.... &\alpha_{1p};\\&\beta_{11},&\beta_{12},.....&\beta_{1q};&\end{matrix}z_1 \right]$ is analytic function of $\alpha_{1i}$ and $\beta_{1j}$ except for simple pole at $\beta_{1j}=0$ or each negative integer for $i=1,2,\ldots,p$ and $j=1,2,\ldots,q$. Therefore, $f_1, f_2$ are both analytic function of $\alpha_{1i},\beta_{1j},\alpha_{2i}$ and $\beta_{2j}$ except  at least one of $\beta_{1j}$ and $\beta_{2j}$ are either zero or negative integer.
Now,
\begin{align*}
   & \frac{\partial f_1}{\partial z}=\frac{1}{2}\left(  {}_{p} F'_{q}\left[\begin{matrix}&\alpha_{11}, &\alpha_{12},\ldots, &\alpha_{1p};\\&\beta_{11},&\beta_{12},\ldots,&\beta_{1q};&\end{matrix}z-i_1z'\right] \right) + \frac{1}{2}\left({}_{p} F'_{q}\left[\begin{matrix}&\alpha_{21}, &\alpha_{22},\ldots, &\alpha_{2p};\\&\beta_{21},&\beta_{22},\ldots,&\beta_{2q};&\end{matrix}z+i_1z'\right]\right), \\
    & \frac{\partial f_1}{\partial z'}=\frac{-i_1}{2}\left(  {}_{p} F'_{q}\left[\begin{matrix}&\alpha_{11}, &\alpha_{12},\ldots, &\alpha_{1p};\\&\beta_{11},&\beta_{12},\ldots,&\beta_{1q};&\end{matrix}z-i_1z' \right]\right) + \frac{i_1}{2}\left({}_{p} F'_{q}\left[\begin{matrix}&\alpha_{21}, &\alpha_{22},\ldots, &\alpha_{2p};\\&\beta_{21},&\beta_{22},\ldots,&\beta_{2q};&\end{matrix}z+i_1z' \right]\right),\\
    &\frac{\partial f_2}{\partial z}=\frac{i_1}{2}\left(  {}_{p} F'_{q}\left[\begin{matrix}&\alpha_{11}, &\alpha_{12},\ldots, &\alpha_{1p};\\&\beta_{11},&\beta_{12},\ldots,&\beta_{1q};&\end{matrix}z-i_1z'\right] \right)  - \frac{i_1}{2}\left({}_{p} F'_{q}\left[\begin{matrix}&\alpha_{21}, &\alpha_{22},\ldots, &\alpha_{2p};\\&\beta_{21},&\beta_{22},\ldots,&\beta_{2q};&\end{matrix}z+i_1z' \right]\right),\\
    \mbox{and}\\
    &\frac{\partial f_2}{\partial z'}=\frac{-i^2_1}{2}\left(  {}_{p} F'_{q}\left[\begin{matrix}&\alpha_{11}, &\alpha_{12},\dots, &\alpha_{1p};\\&\beta_{11},&\beta_{12},\dots,&\beta_{1q};&\end{matrix}z-i_1z' \right]\right) - \frac{i^2_1}{2}\left({}_{p} F'_{q}\left[\begin{matrix}&\alpha_{21}, &\alpha_{22},\dots, &\alpha_{2p};\\&\beta_{21},&\beta_{22},\dots,&\beta_{2q};&\end{matrix}z+i_1z'\right]\right).
\end{align*}
From the above equations, it can be observed that
\begin{align*}
    \frac{\partial f_1}{\partial z}= \frac{\partial f_2}{\partial z'} \quad
      \mbox{and} \quad  \frac{\partial f_1}{\partial z'}=- \frac{\partial f_2}{\partial z}.
\end{align*}
Therefore, $f_1$ and $f_2$ satisfies bicomplex Cauchy-Riemann equation for all $Z\in\mathbb{BC}$. When we consider $\alpha_i$ as a variable and ${}_{p} F_{q}\left[\begin{matrix}&\alpha_1, &\alpha_2,.... &\alpha_{p};\\&\beta_1,&\beta_2,.....&\beta_{p};&\end{matrix}Z\right]=g_1(a_{1i},a_{2i})+i_2g_2(a_{1i},a_{2i})$ then from \eqref{eq:19}, we get
\begin{align*}
    &g_1=\frac{1}{2}\left(  {}_{p} F_{q}\left[\begin{matrix}&\alpha_{11},\ldots,&a_{1i}-i_1a_{2i},\ldots,&\alpha_{1p};\\&\beta_{11},\ldots,&\beta_{1i},\;\ldots,&\beta_{1q};&\end{matrix}z_1 \right] + {}_{p} F_{q}\left[\begin{matrix} &\alpha_{22},\ldots,&a_{1i}+i_1a_{2i},\ldots, &\alpha_{2p};\\&\beta_{21},\ldots,&\beta_{2i},\ldots,&\beta_{2q};&\end{matrix}z_2 \right]\right)\; \mbox{and}\\
   &g_2= \frac{i_1}{2} \left(  {}_{p} F_{q}\left[\begin{matrix}&\alpha_{11},\ldots,&a_{1i}-i_1a_{2i},\ldots,&\alpha_{1p};\\&\beta_{11},\ldots,&\beta_{1i},\;\ldots,&\beta_{1q};&\end{matrix}z_1 \right] - {}_{p} F_{q}\left[\begin{matrix} &\alpha_{22},\ldots,&a_{1i}+i_1a_{2i},\ldots, &\alpha_{2p};\\&\beta_{21},\ldots,&\beta_{2i},\ldots,&\beta_{2q};&\end{matrix}z_2 \right]\right).
\end{align*}
Similarly, we can easily prove $g_1$ and $g_2$ satisfies bicomplex Cauchy-Riemann equation that is
\begin{align*}
    \frac{\partial g_1}{\partial a_{1i}}= \frac{\partial g_2}{\partial a_{2i}} \quad
      \mbox{and} \quad  \frac{\partial g_1}{\partial a_{2i}}=- \frac{\partial g_2}{\partial a_{1i}}.
\end{align*}
 Hence, the bicomplex generalized   hypergeometric function ${}_{p} F_{q}\left[\begin{matrix}&\alpha_1, &\alpha_2,.... &\alpha_{p};\\&\beta_1,&\beta_2,.....&\beta_{q};&\end{matrix}Z\right]$ is $\mathbb{BC}$ holomorphic in $\alpha_i$ and $\beta_j$ for $i=1,2,....p$ and $j=1,2,.....q$ except when at least one of $\beta_{1j}$ and $\beta_{2j}$ are either zero or negative integer.
\end{proof}
\section{Contiguous relations}
In this section several contiguous relations of bicomplex generalized hypergeometric function have been derived.
\begin{theo}\label{eq:th1}
    If  $M=me_1+ne_2\in{\mathbb{BC}}$, where $m,n\in{\mathbb{Z}^+}$ then contiguous relation of the bicomplex generalized hypergeometric functions is given as follows :
    \begin{align*}
    & {}_{p} F_{q}\left[\begin{matrix}&\alpha_1+M, &\alpha_2,\ldots, &\alpha_{p};\\&\beta_1,&\beta_2,\ldots,&\beta_{q};&\end{matrix}Z\right]+{}_{p} F_{q}\left[\begin{matrix}&\alpha_1+\overline{M}, &\alpha_2,\ldots, &\alpha_{p};\\&\beta_1,&\beta_2,\ldots,&\beta_{q};&\end{matrix}Z\right]\\
    &=(\alpha_{1}-1)!\sum_{s=0}^{m}\frac{\binom{m}{s}}{(\alpha_{1}+s-1)!}\frac{\prod\limits_{i=1}^{p}(\alpha_i)_s}{\prod\limits_{j=1}^{q}(\beta_j)_s}Z^s{}_{p} {F}_{q}\left[\begin{matrix}&\alpha_{1}+s, &\alpha_{2}+s,\ldots, &\alpha_{p+}s;\\&\beta_{1}+s,&\beta_{2}+s,\ldots,&\beta_{q}+s;&\end{matrix}Z \right]\\&\quad+(\alpha_{1}-1)!\sum_{t=0}^{n}\frac{\binom{n}{t}}{(\alpha_{1}+t-1)!}\frac{\prod\limits_{i=1}^{p}(\alpha_i)_t}{\prod\limits_{j=1}^{q}(\beta_j)_t}Z^t{}_{p} {F}_{q}\left[\begin{matrix}&\alpha_{1}+t, &\alpha_{2}+t,\ldots, &\alpha_{p+}t;\\&\beta_{1}+t,&\beta_{2}+t,\ldots,&\beta_{q}+t;&\end{matrix}Z \right].
\end{align*}
\end{theo}
\begin{proof}
    Assume that $M=me_1+ne_2\in{\mathbb{BC}}$, where $m,n\in{\mathbb{Z}^+}$ then $\overline{M}=ne_1+me_2$.
Let us derive recurrence relation of the generalized hypergeometric function,
\begin{align}\label{eq:uj}
    {}_{p} F_{q}\left[\begin{matrix}&\alpha_{11}+1, &\alpha_{12},\ldots, &\alpha_{1p};\\&\beta_{11},&\beta_{12},\ldots,&\beta_{1q};&\end{matrix}z_1 \right]&=\sum_{k=0}^{\infty} \frac{(\alpha_{11}+1)_k\prod\limits_{i=2}^{p} (\alpha_{1i})_k}{\prod\limits_{j=1}^{q} (\beta_{1j})_k}\frac{z_1^k}{k!}\nonumber\\
    &=\sum_{k=0}^{\infty} \frac{\prod\limits_{i=1}^{p} (\alpha_{1i})_k}{\prod\limits_{j=1}^{q} (\beta_{1j})_k}\left(\frac{\alpha_{11}+k}{\alpha_{11}}\right)\frac{z_1^k}{k!}\nonumber\\
    &=\sum_{k=0}^{\infty} \frac{\prod\limits_{i=1}^{p} (\alpha_{1i})_k}{\prod\limits_{j=1}^{q} (\beta_{1j})_k}\frac{z_1^k}{k!}+\frac{1}{\alpha_{11}}\sum_{k=1}^{\infty} \frac{\prod\limits_{i=1}^{p} (\alpha_{1i})_k}{\prod\limits_{j=1}^{q} (\beta_{1j})_k}\frac{z_1^k}{(k-1)!}\nonumber\\
    &={}_{p} F_{q}\left[\begin{matrix}&\alpha_{11}, &\alpha_{12},\ldots,&\alpha_{1p};\\&\beta_{11},&\beta_{12},\ldots,&\beta_{1q};&\end{matrix}z_1 \right]+ \frac{z_1}{\alpha_{11}}{}_{p} F'_{q}\left[\begin{matrix}&\alpha_{11}, &\alpha_{12},\ldots,&\alpha_{1p};\\&\beta_{11},&\beta_{12},\ldots,&\beta_{1q};&\end{matrix}z_1 \right].
\end{align}
After replacing  $\alpha_{11}$ by $\alpha_{11}+1$ in the relation \eqref{eq:uj}, we have
\begin{align}\label{eq:ou}
    {}_{p} F_{q}\left[\begin{matrix}&\alpha_{11}+2, &\alpha_{12},\ldots, &\alpha_{1p};\\&\beta_{11},&\beta_{12},\ldots,&\beta_{1q};&\end{matrix}z_1 \right]
    &={}_{p} F_{q}\left[\begin{matrix}&\alpha_{11}, &\alpha_{12},\ldots, &\alpha_{1p};\\&\beta_{11},&\beta_{12},.....&\beta_{1q};&\end{matrix}z_1 \right]+ \frac{2z_1}{\alpha_{11}}{}_{p} F'_{q}\left[\begin{matrix}&\alpha_{11}, &\alpha_{12},\ldots, &\alpha_{1p};\\&\beta_{11},&\beta_{12},\ldots,&\beta_{1q};&\end{matrix}z_1 \right]\nonumber \\&\quad+ \frac{z_1^2}{\alpha_{11}(\alpha_{11}+1)}{}_{p} F''_{q}\left[\begin{matrix}&\alpha_{11}, &\alpha_{12},\ldots, &\alpha_{1p};\\&\beta_{11},&\beta_{12},\ldots,&\beta_{1q};&\end{matrix}z_1 \right].
\end{align}
 Replacing $\alpha_{11}$ by $\alpha_{11}+1$ in the above relation \eqref{eq:ou}, we get
\begin{align*}
     {}_{p} F_{q}\left[\begin{matrix}&\alpha_{11}+3, &\alpha_{12},\ldots &\alpha_{1p};\\&\beta_{11},&\beta_{12},\ldots,&\beta_{1q};&\end{matrix}z_1 \right]
    &={}_{p} F_{q}\left[\begin{matrix}&\alpha_{11}, &\alpha_{12},\ldots, &\alpha_{1p};\\&\beta_{11},&\beta_{12},\ldots,&\beta_{1q};&\end{matrix}z_1 \right]+ \frac{3z_1}{\alpha_{11}}{}_{p} F'_{q}\left[\begin{matrix}&\alpha_{11}, &\alpha_{12},\ldots, &\alpha_{1p};\\&\beta_{11},&\beta_{12},\ldots,&\beta_{1q};&\end{matrix}z_1 \right]\nonumber \\& \quad +\frac{3z_1^2}{\alpha_{11}(\alpha_{11}+1)}{}_{p} F''_{q}\left[\begin{matrix}&\alpha_{11}, &\alpha_{12},\ldots, &\alpha_{1p};\\&\beta_{11},&\beta_{12},\ldots,&\beta_{1q};&\end{matrix}z_1 \right]\\& \quad+\frac{z_1^3}{\alpha_{11}(\alpha_{11}+1)(\alpha_{11}+2)}{}_{p} F''_{q}\left[\begin{matrix}&\alpha_{11}, &\alpha_{12},\ldots, &\alpha_{1p};\\&\beta_{11},&\beta_{12},\ldots,&\beta_{1q};&\end{matrix}z_1 \right].
\end{align*}
Therefore, by mathematical induction we can easily prove that
\begin{align}\label{eq:wlk}
     {}_{p} F_{q}\left[\begin{matrix}&\alpha_{11}+m, &\alpha_{12},\ldots, &\alpha_{1p};\\&\beta_{11},&\beta_{12},\ldots,&\beta_{1q};&\end{matrix}z_1 \right]&=(\alpha_{11}-1)!\sum_{s=0}^{m}\frac{\binom{m}{s}}{(\alpha_{11}+s-1)!}z_1^s{}_{p} {F}_{q}^s\left[\begin{matrix}&\alpha_{11}, &\alpha_{12},\ldots, &\alpha_{1p};\\&\beta_{11},&\beta_{12},\ldots,&\beta_{1q};&\end{matrix}z_1 \right].
\end{align}
Using \eqref{eq:wlk} and \eqref{eq:ab}, we have
\begin{align}\label{eq:sd}
    &{}_{p} F_{q}\left[\begin{matrix}&\alpha_1+M, &\alpha_2,\ldots, &\alpha_{p};\\&\beta_1,&\beta_2,\ldots,&\beta_{q};&\end{matrix}Z\right]\nonumber\\&={}_{p} F_{q}\left[\begin{matrix}&\alpha_{11}+m, &\alpha_{12},\ldots, &\alpha_{1p};\\&\beta_{11},&\beta_{12},\ldots,&\beta_{1q};&\end{matrix}z_1 \right]e_1+{}_{p} F_{q}\left[\begin{matrix}&\alpha_{21}+n, &\alpha_{22},\ldots, &\alpha_{2p};\\&\beta_{21},&\beta_{22},\ldots,&\beta_{2q};&\end{matrix}z_2 \right]e_2\nonumber\\&=(\alpha_{11}-1)!\sum_{s=0}^{m}\frac{\binom{m}{s}}{(\alpha_{11}+s-1)!}z_1^s{}_{p} {F}_{q}^s\left[\begin{matrix}&\alpha_{11}, &\alpha_{12},\ldots, &\alpha_{1p};\\&\beta_{11},&\beta_{12},\ldots,&\beta_{1q};&\end{matrix}z_1 \right]e_1\nonumber\\&\quad+(\alpha_{21}-1)!\sum_{t=0}^{n}\frac{\binom{n}{t}}{(\alpha_{21}+t-1)!}z_2^t{}_{p} {F}_{q}^t\left[\begin{matrix}&\alpha_{21}, &\alpha_{22},\ldots, &\alpha_{2p};\\&\beta_{21},&\beta_{22},\ldots,&\beta_{2q};&\end{matrix}z_2 \right]e_2.
\end{align}
Performing the replacement $M$ by $\overline{M}$ in the above recurrence relation \eqref{eq:sd}, we have
\begin{align}\label{eq:fg}
    {}_{p} F_{q}\left[\begin{matrix}&\alpha_1+\overline{M}, &\alpha_2,\ldots, &\alpha_{p};\\&\beta_1,&\beta_2,\ldots,&\beta_{q};&\end{matrix}Z\right]&=(\alpha_{11}-1)!\sum_{t=0}^{n}\frac{\binom{n}{t}}{(\alpha_{11}+t-1)!}z_1^t{}_{p} {F}_{q}^t\left[\begin{matrix}&\alpha_{11}, &\alpha_{12},\ldots, &\alpha_{1p};\\&\beta_{11},&\beta_{12},\ldots,&\beta_{1q};&\end{matrix}z_1 \right]e_1\nonumber\\&\quad+(\alpha_{21}-1)!\sum_{s=0}^{m}\frac{\binom{m}{s}}{(\alpha_{21}+s-1)!}z_2^s{}_{p} {F}_{q}^s\left[\begin{matrix}&\alpha_{21}, &\alpha_{22},\ldots, &\alpha_{2p};\\&\beta_{21},&\beta_{22},\ldots,&\beta_{2q};&\end{matrix}z_2 \right]e_2.
\end{align}
Finally, using \eqref{eq:sd} and \eqref{eq:fg}, we get
\begin{align*}
   & {}_{p} F_{q}\left[\begin{matrix}&\alpha_1+M, &\alpha_2,\ldots, &\alpha_{p};\\&\beta_1,&\beta_2,\ldots,&\beta_{q};&\end{matrix}Z\right]+{}_{p} F_{q}\left[\begin{matrix}&\alpha_1+\overline{M}, &\alpha_2,\ldots,&\alpha_{p};\\&\beta_1,&\beta_2,\ldots,&\beta_{q};&\end{matrix}Z\right]\\
    &=(\alpha_{1}-1)!\sum_{s=0}^{m}\frac{\binom{m}{s}}{(\alpha_{1}+s-1)!}Z^s{}_{p} {F}_{q}^s\left[\begin{matrix}&\alpha_{1}, &\alpha_{2},\ldots, &\alpha_{p};\\&\beta_{1},&\beta_{2},\ldots,&\beta_{q};&\end{matrix}Z \right]\\&\quad+(\alpha_{1}-1)!\sum_{t=0}^{n}\frac{\binom{n}{t}}{(\alpha_{1}+t-1)!}Z^t{}_{p} {F}_{q}^t\left[\begin{matrix}&\alpha_{1}, &\alpha_{2},\ldots, &\alpha_{p};\\&\beta_{1},&\beta_{2},\ldots,&\beta_{q};&\end{matrix}Z \right]\\
    &=(\alpha_{1}-1)!\sum_{s=0}^{m}\frac{\binom{m}{s}}{(\alpha_{1}+s-1)!}\frac{\prod\limits_{i=1}^{p}(\alpha_i)_s}{\prod\limits_{j=1}^{q}(\beta_j)_s}Z^s{}_{p} {F}_{q}\left[\begin{matrix}&\alpha_{1}+s, &\alpha_{2}+s,\ldots, &\alpha_{p+}s;\\&\beta_{1}+s,&\beta_{2}+s,\ldots,&\beta_{q}+s;&\end{matrix}Z \right]\\&\quad+(\alpha_{1}-1)!\sum_{t=0}^{n}\frac{\binom{n}{t}}{(\alpha_{1}+t-1)!}\frac{\prod\limits_{i=1}^{p}(\alpha_i)_t}{\prod\limits_{j=1}^{q}(\beta_j)_t}Z^t{}_{p} {F}_{q}\left[\begin{matrix}&\alpha_{1}+t, &\alpha_{2}+t,\ldots,&\alpha_{p+}t;\\&\beta_{1}+t,&\beta_{2}+t,\ldots,&\beta_{q}+t;&\end{matrix}Z \right],
\end{align*}
which completes the proof of the theorem.
\end{proof}
\begin{theo}
     Let $M=me_1+ne_2\in{\mathbb{BC}}$, where $m,n\in{\mathbb{Z}^+}$. The following contiguous relation of the bicomplex generalized hypergeometric function holds true.
    \begin{align*}
         & {}_{p} F_{q}\left[\begin{matrix}&\alpha_1-M, &\alpha_2,\ldots, &\alpha_{p};\\&\beta_1,&\beta_2,\ldots,&\beta_{q};&\end{matrix}Z\right]+{}_{p} F_{q}\left[\begin{matrix}&\alpha_1-\overline{M}, &\alpha_2,\ldots, &\alpha_{p};\\&\beta_1\,&\beta_2,\ldots,&\beta_{q};&\end{matrix}Z\right]\\&=\sum_{s=0}^{m}\frac{\binom{m}{s}\prod\limits_{i=2}^{p}(\alpha_{i})_s}{\prod\limits_{j=1}^{q}(\beta_{j})_s}(-Z)^s{}_{p} {F}_{q}\left[\begin{matrix}&\alpha_{1}, &\alpha_{2}+s,\ldots, &\alpha_{p}+s;\\&\beta_{1}+s,&\beta_{2}+s,\ldots,&\beta_{q}+s;&\end{matrix}Z \right]\\& \quad +\sum_{t=0}^{n}\frac{\binom{n}{t}\prod\limits_{i=2}^{p}(\alpha_{i})_t}{\prod\limits_{j=1}^{q}(\beta_{j})_t}(-Z)^t{}_{p} {F}_{q}\left[\begin{matrix}&\alpha_{1}, &\alpha_{2}+t,\ldots, &\alpha_{p}+t;\\&\beta_{1}+t,&\beta_{2}+t,\ldots,&\beta_{q}+t;&\end{matrix}Z \right].
    \end{align*}
\end{theo}
\begin{proof}
    Replacing $\alpha_{11}$ by $(\alpha_{11}-1)$ in \eqref{eq:uj}, we can easily get a contiguous relation of the generalized hypergeometric function as follows :
\begin{align}\label{eq:cn}
     {}_{p} F_{q}\left[\begin{matrix}&\alpha_{11}-1, &\alpha_{12},\ldots, &\alpha_{1p};\\&\beta_{11},&\beta_{12},\ldots,&\beta_{1q};&\end{matrix}z_1 \right]& = {}_{p} F_{q}\left[\begin{matrix}&\alpha_{11}, &\alpha_{12},\ldots, &\alpha_{1p};\\&\beta_{11},&\beta_{12},\ldots,&\beta_{1q};&\end{matrix}z_1 \right]\nonumber\\&\quad-\frac{z_1\prod\limits_{i=2}^{p}(\alpha_{1i})}{\prod\limits_{j=1}^{q}(\beta_{1j})} {}_{p} F_{q}\left[\begin{matrix}&\alpha_{11}, &\alpha_{12}+1,\ldots, &\alpha_{1p}+1;\\&\beta_{11}+1,&\beta_{12}+1,\ldots,&\beta_{1q}+1;&\end{matrix}z_1 \right].
\end{align}
Substitute  $\alpha_{11}$ by $(\alpha_{11}-1)$ in \eqref{eq:cn}, we get
\begin{align}
   & {}_{p} F_{q}\left[\begin{matrix}&\alpha_{11}-2, &\alpha_{12},\ldots, &\alpha_{1p};\nonumber\\&\beta_{11},&\beta_{12},\ldots,&\beta_{1q};&\end{matrix}z_1 \right] \\& = {}_{p} F_{q}\left[\begin{matrix}&\alpha_{11}, &\alpha_{12},\ldots, &\alpha_{1p};\\&\beta_{11},&\beta_{12},\ldots,&\beta_{1q};&\end{matrix}z_1 \right]-\frac{2z_1 \prod\limits_{i=2}^{p}(\alpha_{1i})}{\prod\limits_{j=1}^{q}(\beta_{1j})} {}_{p} F_{q}\left[\begin{matrix}&\alpha_{11}, &\alpha_{12}+1,\ldots, &\alpha_{1p}+1;\\&\beta_{11}+1,&\beta_{12}+1,\ldots,&\beta_{1q}+1;&\end{matrix}z_1 \right]\nonumber\\& \quad+\frac{z_1^2 \prod\limits_{i=2}^{p}{(\alpha_{1i})}_2}{\prod\limits_{j=1}^{q}{(\beta_{1j})}_2} {}_{p} F_{q}\left[\begin{matrix}&\alpha_{11}, &\alpha_{12}+2,\ldots, &\alpha_{1p}+2;\\&\beta_{11}+2,&\beta_{12}+2,\ldots,&\beta_{1q}+2;&\end{matrix}z_1 \right].\nonumber
\end{align}
By principle of mathematical induction, we can easily prove that generalized hypergeometric function satisfies following contiguous relation :
\begin{align}\label{eq:zb}
     &{}_{p} F_{q}\left[\begin{matrix}&\alpha_{11}-m, &\alpha_{12},\ldots, &\alpha_{1p};\\&\beta_{11},&\beta_{12},\ldots,&\beta_{1q};&\end{matrix}z_1 \right]=\sum_{s=0}^{m}\frac{\binom{m}{s}\prod\limits_{i=2}^{p}{(\alpha_{1i})}_s}{\prod\limits_{j=1}^{q}{(\beta_{1j})}_s}(-z_1)^s{}_{p} {F}_{q}^s\left[\begin{matrix}&\alpha_{11}, &\alpha_{12}+s,\ldots, &\alpha_{1p}+s;\\&\beta_{11}+s,&\beta_{12}+s,\ldots,&\beta_{1q}+s;&\end{matrix}z_1 \right].
\end{align}
Using \eqref{eq:zb} and\eqref{eq:ab}, we obtain
\begin{align}\label{eq:jl}
   &{}_{p} F_{q}\left[\begin{matrix}&\alpha_{1}-M, &\alpha_{2},\ldots, &\alpha_{p};\\&\beta_{1},&\beta_{2},\ldots,&\beta_{q};&\end{matrix}Z \right]\nonumber\\&= {}_{p} F_{q}\left[\begin{matrix}&\alpha_{11}-m, &\alpha_{12},\ldots, &\alpha_{1p};\\&\beta_{11},&\beta_{12},\ldots,&\beta_{1q};&\end{matrix}z_1 \right]e_1 +{}_{p} F_{q}\left[\begin{matrix}&\alpha_{21}-n, &\alpha_{22},\ldots, &\alpha_{2p};\\&\beta_{21},&\beta_{22},\ldots,&\beta_{2q};&\end{matrix}z_2 \right]e_2\nonumber\\&=\sum_{s=0}^{m}\frac{\binom{m}{s}\prod\limits_{i=2}^{p}{(\alpha_{1i})}_s}{\prod\limits_{j=1}^{q}{(\beta_{1j})}_s}(-z_1)^s{}_{p} {F}_{q}\left[\begin{matrix}&\alpha_{11}, &\alpha_{12}+s,\ldots, &\alpha_{1p}+s;\\&\beta_{11}+s,&\beta_{12}+s,\ldots,&\beta_{1q}+s;&\end{matrix}z_1 \right]e_1\nonumber\\&\quad+\sum_{t=0}^{n}\frac{\binom{n}{t}\prod\limits_{i=2}^{p}{(\alpha_{2i})}_t}{\prod\limits_{j=1}^{q}{(\beta_{2j})}_t}(-z_2)^t{}_{p} {F}_{q}\left[\begin{matrix}&\alpha_{21}, &\alpha_{22}+t,\ldots, &\alpha_{2p}+t;\\&\beta_{21}+t,&\beta_{22}+t,\ldots,&\beta_{2q}+t;&\end{matrix}z_2 \right]e_2.
\end{align}
Similarly, we obtain
\begin{align}\label{eq:ol}
   {}_{p} F_{q}\left[\begin{matrix}&\alpha_{1}-\overline{M}, &\alpha_{2},\ldots, &\alpha_{p};\\&\beta_{1},&\beta_{2},\ldots,&\beta_{q};&\end{matrix}Z\right]&=\sum_{t=0}^{n}\frac{\binom{n}{t}\prod\limits_{i=2}^{p}{(\alpha_{1i})}_t}{\prod\limits_{j=1}^{q}{(\beta_{1j})}_t}(-z_1)^t{}_{p} {F}_{q}\left[\begin{matrix}&\alpha_{11}, &\alpha_{12}+t,\ldots, &\alpha_{1p}+t;\\&\beta_{11}+t,&\beta_{12}+t,\ldots,&\beta_{1q}+t;&\end{matrix}z_1 \right]\nonumber e_1\\&\quad+\sum_{s=0}^{m}\frac{\binom{m}{s}\prod\limits_{i=2}^{p}{(\alpha_{2i})}_s}{\prod\limits_{j=1}^{q}{(\beta_{2j})}_s}(-z_2)^s{}_{p} {F}_{q}\left[\begin{matrix}&\alpha_{21}, &\alpha_{22}+s,\ldots, &\alpha_{2p}+s;\\&\beta_{21}+s,&\beta_{22}+s,\ldots,&\beta_{2q}+s;&\end{matrix}z_2 \right]e_2.
\end{align}
Combining \eqref{eq:jl} and \eqref{eq:ol}, we get
\begin{align*}
   & {}_{p} F_{q}\left[\begin{matrix}&\alpha_{1}-M, &\alpha_{2},\ldots, &\alpha_{p};\\&\beta_{1},&\beta_{2},\ldots,&\beta_{q};&\end{matrix}Z \right]+{}_{p} F_{q}\left[\begin{matrix}&\alpha_{1}-\overline{M}, &\alpha_{2},\ldots, &\alpha_{p};\\&\beta_{1},&\beta_{2},\ldots,&\beta_{q};&\end{matrix}Z \right]\\&=\sum_{s=0}^{m}\frac{\binom{m}{s}\prod\limits_{i=2}^{p}{(\alpha_{i})}_s}{\prod\limits_{j=1}^{q}{(\beta_{j})}_s}(-Z)^s{}_{p} {F}_{q}\left[\begin{matrix}&\alpha_{1}, &\alpha_{2}+s,\ldots, &\alpha_{p}+s;\\&\beta_{1}+s,&\beta_{2}+s,\ldots,&\beta_{q}+s;&\end{matrix}Z \right]\\& \quad +\sum_{t=0}^{n}\frac{\binom{n}{t}\prod\limits_{i=2}^{p}{(\alpha_{i})}_t}{\prod\limits_{j=1}^{q}{(\beta_{j})}_t}(-Z)^t{}_{p} {F}_{q}\left[\begin{matrix}&\alpha_{1}, &\alpha_{2}+t,\ldots, &\alpha_{p}+t;\\&\beta_{1}+t,&\beta_{2}+t,\ldots,&\beta_{q}+t;&\end{matrix}Z \right].
\end{align*}
Hence the proof is completed.
\end{proof}
\begin{theo}\label{eq:th3}
    Let $M=me_1+ne_2\in{\mathbb{BC}}$, where $m,n\in{\mathbb{Z}^+}$. The following contiguous relation of the bicomplex generalized hypergeometric function hold true.
    \begin{align*}
         & {}_{p} F_{q}\left[\begin{matrix}&\alpha_1, &\alpha_2,\ldots, &\alpha_{p};\\&\beta_1-M,&\beta_2,\dots,&\beta_{q};&\end{matrix}Z\right]+{}_{p} F_{q}\left[\begin{matrix}&\alpha_1, &\alpha_2,\ldots, &\alpha_{p};\\&\beta_1-\overline{M},&\beta_2,\ldots,&\beta_{q};&\end{matrix}Z\right]\\
   & =(\beta_{1}-m-1)!\sum_{s=0}^{m}\frac{\binom{m}{s}}{(\beta_{1}+s-m-1)!}\frac{\prod\limits_{i=1}^{p}{(\alpha_i)}_s}{\prod\limits_{j=1}^{q}{(\beta_j)}_s}Z^s{}_{p} {F}_{q}\left[\begin{matrix}&\alpha_{1}+s, &\alpha_{2}+s,\ldots, &\alpha_{p}+s;\\&\beta_{1}+s,&\beta_{2}+s,\ldots,&\beta_{q}+s;&\end{matrix}Z \right]\\&\quad+(\beta_{1}-n-1)!\sum_{t=0}^{n}\frac{\binom{n}{t}}{(\beta_{1}+t-n-1)!}\frac{\prod\limits_{i=1}^{p}{(\alpha_i)}_t}{\prod\limits_{j=1}^{q}{(\beta_j)}_t}Z^t{}_{p} {F}_{q}\left[\begin{matrix}&\alpha_{1}+t, &\alpha_{2}+t,\ldots, &\alpha_{p}+t;\\&\beta_{1}+t,&\beta_{2}+t,\ldots,&\beta_{q}+t;&\end{matrix}Z \right].
    \end{align*}
\end{theo}
\begin{proof}
   This proof of this theorem is similar to the proof of Theorem \ref{eq:th1}. Let us derived recurrence  relation of the generalized hypergeometric function,
\begin{align}\label{eq:xd}
     {}_{p} F_{q}\left[\begin{matrix}&\alpha_{11}, &\alpha_{12},\ldots, &\alpha_{1p};\\&\beta_{11}-1,&\beta_{12},\ldots,&\beta_{1q};&\end{matrix}z_1 \right]&=\sum_{k=0}^{\infty} \frac{\prod\limits_{i=1}^{p} {(\alpha_{1i})}_k}{\prod\limits_{j=1}^{q} {(\beta_{1j})}_k}\left(\frac{\beta_{11}+k-1}{\beta_{11}-1}\right)\frac{z_1^k}{k!}\nonumber\\
      &={}_{p} F_{q}\left[\begin{matrix}&\alpha_{11}, &\alpha_{12},.... &\alpha_{1p};\\&\beta_{11},&\beta_{12},.....&\beta_{1q};&\end{matrix}z_1 \right]+ \frac{z_1}{\beta_{11}-1}{}_{p} F'_{q}\left[\begin{matrix}&\alpha_{11}, &\alpha_{12},\ldots, &\alpha_{1p};\\&\beta_{11},&\beta_{12},\ldots,&\beta_{1q};&\end{matrix}z_1 \right].
\end{align}
Setting $\beta_{11}$ by $(\beta_{11}-1)$ in the above relation \eqref{eq:xd}, we get
\begin{align*}
   {}_{p} F_{q}\left[\begin{matrix}&\alpha_{11}, &\alpha_{12},\ldots, &\alpha_{1p};\\&\beta_{11}-2,&\beta_{12},\dots,&\beta_{1q};&\end{matrix}z_1 \right]&= {}_{p} F_{q}\left[\begin{matrix}&\alpha_{11}, &\alpha_{12},\ldots, &\alpha_{1p};\\&\beta_{11},&\beta_{12},\ldots,&\beta_{1q};&\end{matrix}z_1 \right]+ \frac{2z_1}{\beta_{11}-2}{}_{p} F'_{q}\left[\begin{matrix}&\alpha_{11}, &\alpha_{12},\ldots, &\alpha_{1p};\\&\beta_{11},&\beta_{12},\ldots,&\beta_{1q};&\end{matrix}z_1 \right]\\&\quad+\frac{z_1^2}{(\beta_{11}-1)(\beta_{11}-2)}{}_{p} F''_{q}\left[\begin{matrix}&\alpha_{11}, &\alpha_{12},\ldots, &\alpha_{1p};\\&\beta_{11},&\beta_{12},\ldots,&\beta_{1q};&\end{matrix}z_1 \right].
\end{align*}
 By mathematical induction, we can easily show the generalized hypergeometric function satisfies recurrence relation
\begin{align}\label{eq:iu}
  & {}_{p} F_{q}\left[\begin{matrix}&\alpha_{11}, &\alpha_{12},\ldots, &\alpha_{1p};\\&\beta_{11}-m,&\beta_{12},\ldots,&\beta_{1q};&\end{matrix}z_1 \right]\nonumber\\&=(\beta_{11}-m-1)!\sum_{s=0}^{m}\frac{\binom{m}{s}}{(\beta_{11}+s-m-1)!}z_1^s{}_{p} {F}_{q}^s\left[\begin{matrix}&\alpha_{11}, &\alpha_{12},\ldots, &\alpha_{1p};\\&\beta_{11},&\beta_{12},\ldots,&\beta_{1q};&\end{matrix}z_1 \right].
\end{align}
Using \eqref{eq:ab} and \eqref{eq:iu}, we have
\begin{align}\label{eq:kj}
  &{}_{p} F_{q}\left[\begin{matrix}&\alpha_1, &\alpha_2,\ldots, &\alpha_{p};\\&\beta_1-M,&\beta_2,\ldots,&\beta_{q};&\end{matrix}Z\right]\nonumber\\&= {}_{p} F_{q}\left[\begin{matrix}&\alpha_{11}, &\alpha_{12},\ldots, &\alpha_{1p};\\&\beta_{11}-m,&\beta_{12},\ldots,&\beta_{1q};&\end{matrix}z_1 \right]e_1 +{}_{p} F_{q}\left[\begin{matrix}&\alpha_{21}, &\alpha_{22},\ldots, &\alpha_{2p};\\&\beta_{21}-n,&\beta_{22},\ldots,&\beta_{2q};&\end{matrix}z_2 \right]e_2\nonumber\\&=(\beta_{11}-m-1)!\sum_{s=0}^{m}\frac{\binom{m}{s}}{(\beta_{11}+s-m-1)!}z_1^s{}_{p} {F}_{q}^s\left[\begin{matrix}&\alpha_{11}, &\alpha_{12},\ldots, &\alpha_{1p};\\&\beta_{11},&\beta_{12},\ldots,&\beta_{1q};&\end{matrix}z_1 \right]e_1\nonumber\\&\quad+(\beta_{21}-n-1)!\sum_{t=0}^{n}\frac{\binom{n}{t}}{(\beta_{21}+t-n-1)!}z_2^t{}_{p} {F}_{q}^t\left[\begin{matrix}&\alpha_{21}, &\alpha_{22},\ldots, &\alpha_{2p};\\&\beta_{21},&\beta_{22},\ldots,&\beta_{2q};&\end{matrix}z_2 \right]e_2.
\end{align}
Substituting $M$ by $\overline{M}$ in \eqref{eq:kj}, we get
\begin{align}\label{eq:fh}
     &{}_{p} F_{q}\left[\begin{matrix}&\alpha_1, &\alpha_2,\ldots, &\alpha_{p};\\&\beta_1-\overline{M},&\beta_2,\ldots,&\beta_{q};&\end{matrix}Z\right]\nonumber\\&= (\beta_{11}-n-1)!\sum_{t=0}^{n}\frac{\binom{n}{t}}{(\beta_{11}+t-n-1)!}z_1^t{}_{p} {F^t}_{q}\left[\begin{matrix}&\alpha_{11}, &\alpha_{12},\ldots, &\alpha_{1p};\\&\beta_{11},&\beta_{12},\ldots,&\beta_{1q};&\end{matrix}z_1 \right]e_1\nonumber\\&\quad+(\beta_{21}-m-1)!\sum_{s=0}^{m}\frac{\binom{m}{s}}{(\beta_{21}+s-m-1)!}z_2^s{}_{p} {F}_{q}^s\left[\begin{matrix}&\alpha_{21}, &\alpha_{22},\ldots, &\alpha_{2p};\\&\beta_{21},&\beta_{22},\ldots,&\beta_{2q};&\end{matrix}z_2 \right]e_2.
\end{align}
Using \eqref{eq:kj} and \eqref{eq:fh}, we obtain
\begin{align*}
    &{}_{p} F_{q}\left[\begin{matrix}&\alpha_1, &\alpha_2,\ldots, &\alpha_{p};\\&\beta_1-M,&\beta_2,\ldots,&\beta_{q};&\end{matrix}Z\right]+ {}_{p} F_{q}\left[\begin{matrix}&\alpha_1, &\alpha_2,\ldots, &\alpha_{p};\\&\beta_1-\overline{M},&\beta_2,\ldots,&\beta_{q};&\end{matrix}Z\right]\\&=(\beta_{1}-m-1)!\sum_{s=0}^{m}\frac{\binom{m}{s}}{(\beta_{1}+s-m-1)!}Z^s{}_{p} {F}_{q}^s\left[\begin{matrix}&\alpha_{1}, &\alpha_{2},\ldots, &\alpha_{p};\\&\beta_{1},&\beta_{2},\ldots,&\beta_{q};&\end{matrix}Z \right]\\&\quad+(\beta_{1}-n-1)!\sum_{t=0}^{n}\frac{\binom{n}{t}}{(\beta_{1}+t-n-1)!}Z^t{}_{p} {F}_{q}^t\left[\begin{matrix}&\alpha_{1}, &\alpha_{2},\ldots, &\alpha_{p};\\&\beta_{1},&\beta_{2},\ldots,&\beta_{q};&\end{matrix}Z \right]\\
   & =(\beta_{1}-m-1)!\sum_{s=0}^{m}\frac{\binom{m}{s}}{(\beta_{1}+s-m-1)!}\frac{\prod\limits_{i=1}^{p}{(\alpha_i)}_s}{\prod\limits_{j=1}^{q}{(\beta_j)}_s}Z^s{}_{p} {F}_{q}\left[\begin{matrix}&\alpha_{1}+s, &\alpha_{2}+s,\ldots, &\alpha_{p}+s;\\&\beta_{1}+s,&\beta_{2}+s,\ldots,&\beta_{q}+s;&\end{matrix}Z \right]\\&\quad+(\beta_{1}-n-1)!\sum_{t=0}^{n}\frac{\binom{n}{t}}{(\beta_{1}+t-n-1)!}\frac{\prod\limits_{i=1}^{p}{(\alpha_i)}_t}{\prod\limits_{j=1}^{q}{(\beta_j)}_t}Z^t{}_{p} {F}_{q}\left[\begin{matrix}&\alpha_{1}+t, &\alpha_{2}+t,\ldots, &\alpha_{p}+t;\\&\beta_{1}+t,&\beta_{2}+t,\ldots,&\beta_{q}+t;&\end{matrix}Z \right],
\end{align*}
which completes the proof of this theorem.
\end{proof}
\begin{theo}
    If  $M=me_1+ne_2\in{\mathbb{BC}}$, where $m,n\in{\mathbb{Z}^+}$ then another contiguous relation of the bicomplex generalized hypergeometric functions is given by
    \begin{align*}
     &{}_{p} F_{q}\left[\begin{matrix}&\alpha_1, &\alpha_2,\ldots, &\alpha_{p};\\&\beta_1+M,&\beta_2,\ldots,&\beta_{q};&\end{matrix}Z\right]+{}_{p} F_{q}\left[\begin{matrix}&\alpha_1, &\alpha_2,\ldots, &\alpha_{p};\\&\beta_1+\overline{M},&\beta_2,\ldots,&\beta_{q};&\end{matrix}Z\right]\\&=2 {}_{p} F_{q}\left[\begin{matrix}&\alpha_1, &\alpha_2,\ldots,&\alpha_{p};\\&\beta_1,&\beta_2,\ldots,&\beta_{q};&\end{matrix}Z\right]-Z\sum_{s=0}^{m}\frac{\prod\limits_{i=1}^{p}\alpha_{i}}{{(\beta_{1}+s-1)}_2 \prod\limits_{j=2}^{q}\beta_{j}}{}_{p} {F}_{q}\left[\begin{matrix}&\alpha_{1}+1, &\alpha_{2}+1,\ldots, &\alpha_{p}+1;\\&\beta_{1}+s+1,&\beta_{2}+1,\ldots,&\beta_{q}+1;&\end{matrix}Z \right]\\&\quad-Z\sum_{t=0}^{n}\frac{\prod\limits_{i=1}^{p}\alpha_{i}}{{(\beta_{1}+t-1)}_2 \prod\limits_{j=2}^{q}\beta_{j}}{}_{p} {F}_{q}\left[\begin{matrix}&\alpha_{1}+1, &\alpha_{2}+1,\ldots, &\alpha_{p}+1;\\&\beta_{1}+t+1,&\beta_{2}+1,\ldots,&\beta_{q}+1;&\end{matrix}Z \right].
\end{align*}
\end{theo}
\begin{proof}
    A contiguous relation of the generalized hypergeometric function can be obtained by replacing $\beta_{11}$ by $\beta_{11}+1$ in \eqref{eq:xd} as follows :
\begin{align}\label{eq:bnm}
     &{}_{p} F_{q}\left[\begin{matrix}&\alpha_{11}, &\alpha_{12},\ldots, &\alpha_{1p};\\&\beta_{11}+1,&\beta_{12},\ldots,&\beta_{1q};&\end{matrix}z_1 \right]= {}_{p} F_{q}\left[\begin{matrix}&\alpha_{11}, &\alpha_{12},\ldots, &\alpha_{1p};\\&\beta_{11},&\beta_{12},\ldots,&\beta_{1q};&\end{matrix}z_1 \right]\nonumber\\
    &\hspace{40pt}-\frac{z_1\prod\limits_{i=1}^{p}\alpha_{1i}}{\beta_{11}(\beta_{11}+1)\prod\limits_{j=2}^{q}\beta_{1j}}{}_{p} F_{q}\left[\begin{matrix}&\alpha_{11}+1, &\alpha_{12}+1,\ldots, &\alpha_{1p}+1;\\&\beta_{11}+2,&\beta_{12}+1,\ldots,&\beta_{1q}+1;&\end{matrix}z_1 \right].
\end{align}
Substituting $\beta_{11}$ by $\beta_{11}+1$ in \eqref{eq:bnm}, we have
\begin{align*}
     &{}_{p} F_{q}\left[\begin{matrix}&\alpha_{11}, &\alpha_{12},\ldots, &\alpha_{1p};\\&\beta_{11}+2,&\beta_{12},\ldots,&\beta_{1q};&\end{matrix}z_1 \right]= {}_{p} F_{q}\left[\begin{matrix}&\alpha_{11}, &\alpha_{12},\ldots, &\alpha_{1p};\\&\beta_{11}+1,&\beta_{12},\ldots,&\beta_{1q};&\end{matrix}z_1 \right]\\&\hspace{40pt}-\frac{z_1\prod\limits_{i=1}^{p}\alpha_{1i}}{(\beta_{11}+1)(\beta_{11}+2)\prod\limits_{j=2}^{q}\beta_{1j}}{}_{p} F_{q}\left[\begin{matrix}&\alpha_{11}+1, &\alpha_{12}+1,\ldots, &\alpha_{1p}+1;\\&\beta_{11}+3,&\beta_{12}+1,\ldots,&\beta_{1q}+1;&\end{matrix}z_1 \right].  \end{align*}
By mathematical induction, we can easily show the generalized hypergeometric function satisfies recurrence relation
\begin{align}\label{eq:ijn}
     &{}_{p} F_{q}\left[\begin{matrix}&\alpha_{11}, &\alpha_{12},\ldots, &\alpha_{1p};\\&\beta_{11}+m,&\beta_{12},\ldots,&\beta_{1q};&\end{matrix}z_1 \right]= {}_{p} F_{q}\left[\begin{matrix}&\alpha_{11}, &\alpha_{12},\ldots, &\alpha_{1p};\\&\beta_{11},&\beta_{12},\ldots,&\beta_{1q};&\end{matrix}z_1 \right]\nonumber\\&\hspace{40pt}-z_1\sum_{s=1}^{m}\frac{\prod\limits_{i=1}^{p}\alpha_{1i}}{{(\beta_{11}+s-1)}_2 \prod\limits_{j=2}^{q}\beta_{1j}}{}_{p} {F}_{q}\left[\begin{matrix}&\alpha_{11}+1, &\alpha_{12}+1,\ldots, &\alpha_{1p}+1;\\&\beta_{11}+s+1,&\beta_{12}+1,\ldots,&\beta_{1q}+1;&\end{matrix}z_1 \right].
\end{align}
Using \eqref{eq:ab} and \eqref{eq:ijn}, we get
\begin{align}\label{eq:vb}
    &{}_{p} F_{q}\left[\begin{matrix}&\alpha_1, &\alpha_2,\ldots, &\alpha_{p};\\&\beta_1+M,&\beta_2,\ldots,&\beta_{q};&\end{matrix}Z\right]=
    {}_{p} F_{q}\left[\begin{matrix}&\alpha_1, &\alpha_2,\ldots,&\alpha_{p};\\&\beta_1,&\beta_2,\ldots,&\beta_{q};&\end{matrix}Z\right]\nonumber\\&\hspace{40pt}-z_1\sum_{s=1}^{m}\frac{\prod\limits_{i=1}^{p}\alpha_{1i}}{{(\beta_{11}+s-1)}_2 \prod\limits_{j=2}^{q}\beta_{1j}}{}_{p} {F}_{q}\left[\begin{matrix}&\alpha_{11}+1, &\alpha_{12}+1,\ldots, &\alpha_{1p}+1;\\&\beta_{11}+s+1,&\beta_{12}+1,\ldots,&\beta_{1q}+1;&\end{matrix}z_1 \right]e_1\nonumber\\&\hspace{40pt}-z_2\sum_{t=1}^{n}\frac{\prod\limits_{i=1}^{p}\alpha_{2i}}{{(\beta_{21}+t-1)}_2 \prod\limits_{j=2}^{q}\beta_{1j}}{}_{p} {F}_{q}\left[\begin{matrix}&\alpha_{21}+1, &\alpha_{22}+1,\ldots, &\alpha_{2p}+1;\\&\beta_{21}+t+1,&\beta_{22}+1,\ldots,&\beta_{2q}+1;&\end{matrix}z_2 \right]e_2.
\end{align}
Similarly, we have
\begin{align}\label{eq:hn}
    &{}_{p} F_{q}\left[\begin{matrix}&\alpha_1, &\alpha_2,\ldots, &\alpha_{p};\\&\beta_1+\overline{M},&\beta_2,\ldots,&\beta_{q};&\end{matrix}Z\right]=
    {}_{p} F_{q}\left[\begin{matrix}&\alpha_1, &\alpha_2,\ldots, &\alpha_{p};\\&\beta_1,&\beta_2,\ldots,&\beta_{q};&\end{matrix}Z\right]\nonumber\\&\hspace{40pt}-z_1\sum_{t=1}^{n}\frac{\prod\limits_{i=1}^{p}\alpha_{1i}}{{(\beta_{11}+t-1)}_2 \prod\limits_{j=2}^{q}\beta_{1j}}{}_{p} {F}_{q}\left[\begin{matrix}&\alpha_{11}+1, &\alpha_{12}+1,\ldots, &\alpha_{1p}+1;\\&\beta_{11}+t+1,&\beta_{12}+1,\ldots,&\beta_{1q}+1;&\end{matrix}z_1 \right]e_1\nonumber\\&\hspace{40pt}-z_2\sum_{s=1}^{m}\frac{\prod\limits_{i=1}^{p}\alpha_{2i}}{{(\beta_{21}+s-1)}_2 \prod\limits_{j=2}^{q}\beta_{2j}}{}_{p} {F}_{q}\left[\begin{matrix}&\alpha_{21}+1, &\alpha_{22}+1,\ldots, &\alpha_{2p}+1;\\&\beta_{21}+s+1,&\beta_{22}+1,\ldots,&\beta_{2q}+1;&\end{matrix}z_2 \right]e_2.
\end{align}
Combining \eqref{eq:vb} and \eqref{eq:hn}, we obtain
\begin{align*}
     &{}_{p} F_{q}\left[\begin{matrix}&\alpha_1, &\alpha_2,\ldots, &\alpha_{p};\\&\beta_1+M,&\beta_2,\ldots,&\beta_{q};&\end{matrix}Z\right]+{}_{p} F_{q}\left[\begin{matrix}&\alpha_1, &\alpha_2,\ldots, &\alpha_{p};\\&\beta_1+\overline{M},&\beta_2,\ldots,&\beta_{q};&\end{matrix}Z\right]\\&=2 {}_{p} F_{q}\left[\begin{matrix}&\alpha_1, &\alpha_2,\ldots, &\alpha_{p};\\&\beta_1,&\beta_2,\ldots,&\beta_{q};&\end{matrix}Z\right]-Z\sum_{s=1}^{m}\frac{\prod\limits_{i=1}^{p}\alpha_{i}}{{(\beta_{1}+s-1)}_2 \prod\limits_{j=2}^{q}\beta_{j}}{}_{p} {F}_{q}\left[\begin{matrix}&\alpha_{1}+1, &\alpha_{2}+1,\ldots, &\alpha_{p}+1;\\&\beta_{1}+s+1,&\beta_{2}+1,\ldots,&\beta_{q}+1;&\end{matrix}Z \right]\\&\quad-Z\sum_{t=1}^{n}\frac{\prod\limits_{i=1}^{p}\alpha_{i}}{{(\beta_{1}+t-1)} _2 \prod\limits_{j=2}^{q}\beta_{j}}{}_{p} {F}_{q}\left[\begin{matrix}&\alpha_{1}+1, &\alpha_{2}+1,\ldots, &\alpha_{p}+1;\\&\beta_{1}+t+1,&\beta_{2}+1,\ldots,&\beta_{q}+1;&\end{matrix}Z \right].
\end{align*}
Hence the proof is completed.
\end{proof}
\section{Bicomplex generalized hypergeometric differential equation}
In this section, we derive a differential equation which satisfies bicomplex generalized hypergeometric function. Moreover, we establish the differential equation for specific value $p$ and $q$ whose solution is bicomplex confluent hypergeometric function and bicomplex hypergeometric function.
\begin{theo}\label{eq:th7}
    Bicomplex generalized hypergeometric function $Y(Z)={}_{p} F_{q}\left[\begin{matrix}&\alpha_1, &\alpha_2,\ldots, &\alpha_{p};\\&\beta_1,&\beta_2,\ldots,&\beta_{q};&\end{matrix}Z\right]$ is a solution of the following differential equation :
    \begin{align}\label{eq:jkl}
        \frac{d}{dZ} \prod\limits_{j=1}^{q}\left(Z\frac{d}{dZ}+\beta_j-1\right)Y(Z)-\prod\limits_{i=1}^{p}\left(Z\frac{d}{dZ}+\alpha_i\right)Y(Z)=0.
    \end{align}
\end{theo}
\begin{proof}
    Setting $M=1$ in \eqref{eq:sd} of Theorem \ref{eq:th1}, we obtain the following differential equation
    \begin{align*}
        \left(Z\frac{d}{dZ}+\alpha_1\right){}_{p} F_{q}\left[\begin{matrix}&\alpha_1, &\alpha_2,\ldots, &\alpha_{p};\\&\beta_1,&\beta_2,\ldots,&\beta_{q};&\end{matrix}Z\right]=(\alpha_1){}_{p} F_{q}\left[\begin{matrix}&\alpha_1+1, &\alpha_2,\ldots,&\alpha_{p};\\&\beta_1,&\beta_2,\ldots,&\beta_{q};&\end{matrix}Z\right],
    \end{align*}
    which implies
    \begin{align*}
      \prod\limits_{i=1}^{p}\left(Z\frac{d}{dZ}+\alpha_i\right){}_{p} F_{q}\left[\begin{matrix}&\alpha_1, &\alpha_2,\ldots, &\alpha_{p};\\&\beta_1,&\beta_2,\ldots,&\beta_{q};&\end{matrix}Z\right]=\prod\limits_{i=1}^{p}(\alpha_i){}_{p} F_{q}\left[\begin{matrix}&\alpha_1+1, &\alpha_2+1,\ldots, &\alpha_{p}+1;\\&\beta_1,&\beta_2,\ldots,&\beta_{q};&\end{matrix}Z\right] .
    \end{align*}
   Setting $M=1$ in \eqref{eq:kj} of Theorem \ref{eq:th3}, we get
   \begin{align*}
     \left(Z\frac{d}{dZ}+\beta_1-1\right){}_{p} F_{q}\left[\begin{matrix}&\alpha_1, &\alpha_2,\ldots, &\alpha_{p};\\&\beta_1,&\beta_2,\ldots,&\beta_{q};&\end{matrix}Z\right]=(\beta_1-1){}_{p} F_{q}\left[\begin{matrix}&\alpha_1, &\alpha_2,\ldots, &\alpha_{p};\\&\beta_1-1,&\beta_2,\ldots,&\beta_{q};&\end{matrix}Z\right],
   \end{align*}
   which yields
   \begin{align}\label{eq:am}
       \prod\limits_{j=1}^{q}\left(Z\frac{d}{dZ}+\beta_j-1\right){}_{p} F_{q}\left[\begin{matrix}&\alpha_1, &\alpha_2,\ldots,&\alpha_{p};\\&\beta_1,&\beta_2,\ldots,&\beta_{q};&\end{matrix}Z\right]=\prod\limits_{j=1}^{q}(\beta_j-1){}_{p} F_{q}\left[\begin{matrix}&\alpha_1, &\alpha_2,\ldots, &\alpha_{p};\\&\beta_1-1,&\beta_2-1,\ldots,&\beta_{q}-1;&\end{matrix}Z\right] .
   \end{align}
   Now differentiate both sides of the equation \eqref{eq:am} we respect to $Z$, we have
   \begin{align*}
      \frac{d}{dZ}\prod\limits_{j=1}^{q}\left(Z\frac{d}{dZ}+\beta_j-1\right)Y(Z)&=\prod\limits_{j=1}^{q}(\beta_j-1){}_{p} F'_{q}\left[\begin{matrix}&\alpha_1, &\alpha_2,\ldots, &\alpha_{p};\\&\beta_1-1,&\beta_2-1,\ldots,&\beta_{q}-1;&\end{matrix}Z\right]\\
       &=\prod\limits_{i=1}^{p}(\alpha_i){}_{p} F_{q}\left[\begin{matrix}&\alpha_1+1, &\alpha_2+1,\ldots, &\alpha_{p}+1;\\&\beta_1,&\beta_2,\ldots,&\beta_{q};&\end{matrix}Z\right] \\
       &=\prod\limits_{i=1}^{p}\left(Z\frac{d}{dZ}+\alpha_i\right)Y(Z),
   \end{align*}
  which completes the proof of the theorem.
\end{proof}
\begin{remark}
    Setting $p=q=1$ in \eqref{eq:jkl}, we obtain bicomplex confluent hypergeometric function $Y(Z)={}_1F_1(\alpha_1;\beta_1;Z)$ satisfies the differential equation
    \begin{align*}
        Z\frac{d^2Y}{dZ^2}+(\beta_1-Z)\frac{dY}{dZ}-\alpha_1Y=0.
    \end{align*}
\end{remark}
\begin{remark}
    Setting $p=2$ and $q=1$ in \eqref{eq:jkl}, we get bicomplex hypergeometric function $y(Z)={}_2F_1(\alpha_1,\alpha_2;\beta_1;Z)$ as a solution of the differential equation
    \begin{align*}
        Z(1-Z)\frac{d^2Y}{dZ^2}+[\beta_1-(\alpha_1+\alpha_2+1)Z]\frac{dY}{dZ}-\alpha_1\alpha_2Y=0.
    \end{align*}
\end{remark}
 \section{Applications of bicomplex generalized hypergeometric functions in coherent states}
Coherent States (CS), dating back to the birth of quantum mechanics, were discovered by E. Schr\"{o}dinger \cite{sch} in course of searching states of quantum harmonic oscillator which are closest possible to their classical counterpart (although he had not used the name `coherent states'). In 1963 R. J. Glauber \cite{Glauber} had extended Schr\"{o}dinger's approach to quantum electrodynamics and used the term 'coherent states'. CSs are nowadays widely applied in different fields in physics like signals and quantum information system, quantum optics etc. There are many significant books and articles on CSs and its applications (among these we refer \cite{B1,B2}).

For the problem of harmonic oscillator \cite{Glauber}, CS are superpositions of the eigenstates $\mid n>$ associated with the number operator $N=B^\dag B$ defined in terms of the annihilation operator $B$ and  creation operator $B^\dag$  satisfying the following commutation relations
\begin{equation}\label{1}
[N,B]=-B, [N,B^\dag]=B^\dag \mbox{ and }[B,B^\dag]=1\nonumber
\end{equation}
and CS are defined as the eigenstates $\mid n>$ of $B$.

Further generalizing the notion of CSs of harmonic oscillator, CS associated a class of nonlinear deformed oscillators \cite{daskal} can be constructed in terms of annihilation and creation operators
 \begin{equation}\label{2}
 A=Bf(N), A^\dag=f(N)B^\dag\nonumber
\end{equation}
satisfying deformed commutation relations
\begin{equation}\label{3}
[N,A]=-A, [N,A^\dag]=A^\dag \mbox{ and }[A,A^\dag]=(N+1)f^2(N+1)-Nf^2(N)\nonumber
\end{equation}
where $f$ is a hermitian operator-valued function of number operator and CS can be defined as the eigenstates of $A$. Following this an interesting category of CS, namely hypergeometric coherent states, were introduced in \cite{AS,Quesne} as a large class of holomorphic eigenstates of suitably defined annihilation operators whose normalization functions can be expressed in terms of generalized hypergeometric functions.

Aiming at the generalizations of the results of \cite{AS} to bicomplex setting, in the following we are investigating the coherent states whose normalization functions are defined in terms of bicomplex generalized hypergeometric functions. First, we define bicomplex generalized hypergeometric states (BGHS) in an infinite dimensional bicomplex Hilbert space \cite{Lav} of the Fock states
\begin{equation}\label{fock}
\mid n>=\sum_{s=1}^{2}\mid n_s> e_s,\quad n_s=0,1,2,...\nonumber
\end{equation}
as
\begin{eqnarray}\label{cs1}
\mid p;q;Z> &\equiv &\mid \alpha_1,...,\alpha_p;\beta_1,...,\beta_q;Z>\nonumber\\
&=& \frac{1}{\sqrt{\mathcal{N}^{(p,q)}(\mid Z\mid_h^2)}}\sum_{n=0}^\infty \frac{Z^n}{\sqrt{\rho^{(p,q)}(n)}}\mid n>
\end{eqnarray}
where the normalization functions are defined in terms of bicomplex generalized hypergeometric functions as follows
\begin{eqnarray}\label{cs3}
  \mathcal{N}^{(p,q)}(\zeta)&=& _{p}F_q(\alpha_1,...,\alpha_p;\beta_1,...,\beta_q;\zeta)\nonumber\\
  &=& \sum_{s=1}^{2}\left[\sum_{n_s=0}^\infty\frac{\prod_{j=1}^p(\alpha_{sj})_{n_s}}{\prod_{j=1}^q(\beta_{sj})_{n_s}}.\frac{\zeta_s^{n_s}}{n_s!}\right]e_s\nonumber
\end{eqnarray}
 for $\zeta=\sum_{s=1}^2\zeta_s e_s=\mid Z \mid_h^2$. Since $\mid Z \mid_h=\sum_{s=1}^2 \mid Z_s\mid e_s$, we can obtain $\zeta_s=\mid Z_s\mid^2, s=1,2$.

The parameter function $\rho^{(p,q)}(n)$ in (\ref{cs1}) is defined by
\begin{eqnarray}\label{cs2}
\rho^{(p,q)}(n) &\equiv & \rho^{(p,q)}(\alpha_1,...,\alpha_p;\beta_1,...,\beta_q;n)\nonumber\\
&=& \sum_{s=1}^2 \Gamma(n_s+1)\frac{\prod_{j=1}^q (\beta_{sj})_{n_s}}{\prod_{j=1}^p (\alpha_{sj})_{n_s}}e_s.
\end{eqnarray}
For the validity of $\rho^{(p,q)}(n)$ in the expression (\ref{cs2}), we have no choice other than to assume that $\alpha_{sj}, s=1,2$ are neither zero nor negative integer.

Now introducing a bicomplex valued function
\begin{equation*}
f^{(p,q)}(m)=\sum_{s=1}^2\sqrt{(m_s+1)\frac{\prod_{j=1}^q {\beta_{sj}}}{\prod_{i=1}^p {\alpha_{sj}}}}e_s,\quad m=\sum_{s=1}^2 m_s e_s
\end{equation*}
so that
\begin{equation}\label{cs2a}
\prod_{m=0}^{n-1}f^{(p,q)}(m)=\sum_{s=1}^2\sqrt{\Gamma(n_s+1)\frac{\prod_{j=1}^q (\beta_{sj})_{n_s}}{\prod_{i=1}^p (\alpha_{sj})_{n_s}}}e_s=\sqrt{\rho^{(p,q)}(n)}.\nonumber
\end{equation}
A recurrence relation on $\rho^{(p,q)}(n)$ can be obtain as follows :
\begin{equation}\label{cs2b}
\rho^{(p,q)}(n+1)=\rho^{(p,q)}(n)\left(f^{(p,q)}(n)\right)^2,
\end{equation}
with seed value $\rho^{(p,q)}(0)=\sum_{s=1}^2 e_s=1$. For the above, $\rho^{(p,q)}(n)$ (correspondingly $f^{(p,q)}(n)$ also) requires to be a strictly positive hyperbolic number and so we must have another pair of restrictions on $\alpha_{si},\beta_{sj},i=1,2,...,p; j=1,2,...,q; s=1,2$ namely
\begin{equation*}
\frac{\prod_{j=1}^q \beta_{sj}}{\prod_{i=1}^p \alpha_{si}}>0,\quad s=1,2.
\end{equation*}

We can determine the scalar product of two BGHS with identical parameters with the help of normalized functions as
\begin{equation}\label{cs4}
<p;q;Z \mid p;q;Z'>=\frac{\mathcal{N}^{(p,q)}(Z*Z')}{\sqrt{\mathcal{N}^{(p,q)}(\mid Z\mid_h^2)}\sqrt{\mathcal{N}^{(p,q)}(\mid Z'\mid_h^2)}}
\end{equation}
which implies the BGHS are normalized but not orthogonal. Moreover, the expression (\ref{cs4}) will be well defined if the bicomplex generalized hypergeometric functions involved in the scalar product are (unconditionally) convergent. Following the Theorem 2.2, the normalized states defined in (\ref{cs1}) are thus holomorphic only if either $p\leq q$ or $p=q+1$ and $\mathfrak{R}\left(\sum_{j=1}^q b_{1j}-\sum_{i=1}^p a_{1i}\right)>\mid \mathfrak{I}\left(\sum_{j=1}^q b_{2j}-\sum_{i=1}^p a_{2i}\right)\mid$.

In the following we introduce ladder operators for bicomplex shifted oscillators of creation and annihilation types, the later having the BGHS as eigenstates. Thereby, those BGHS serve the role of coherent states of the corresponding shifted oscillators.

We define bicomplex generalized hypergeometric annihilation and creation operators, those are densely defined in the infinite dimensional Hilbert space of the Fock states $\mid n>, n=0,1,2,...$, as
\begin{eqnarray}
\mathcal{A}_-^{(p,q)}\equiv\mathcal{A}_-&=&\sum_{n=0}^\infty f^{(p,q)}(n)\mid n><n+1\mid\label{csn1}\nonumber\\
\mathcal{A}_+^{(p,q)}\equiv\mathcal{A}_+&=&\sum_{n=0}^\infty f^{(p,q)}(n)\mid n+1><n\mid\label{csn2}\nonumber
\end{eqnarray}
where $\mathcal{A}_+ = \mathcal{A}_-^\dagger$ that is $\mathcal{A}_+$ is adjoint operator\footnote{Bicomplex adjoint operator $A^\dagger$ of $A$ is the linear operator from a $\mathbb{BC}$ module  to  itself defined by
$\langle\psi,A\phi\rangle=\langle A^\dagger\psi,\phi\rangle
$ for any bicomplex valued functions $\phi,\psi\in\mathbb{BC}$-module. We refer \cite{ban} for detail in this purpose.} of $\mathcal{A}_-$.

Following the orthogonality and completeness relations of the Fock states, which are
\begin{equation}\label{csn3}
<n\mid m>=\sum_{s=1}^{2}<n_s\mid m_s>e_s=\sum_{s=1}^{2}\delta_{n_s m_s}e_s\nonumber
\end{equation}
and
\begin{equation}\label{csn4}
\sum_{n=0}^{\infty}\mid n><n\mid=\sum_{s=1}^{2}\sum_{n_s=0}^{\infty}\mid n_s><n_s\mid e_s=\sum_{s=1}^{2}e_s=1\nonumber
\end{equation}
we can obtain the following relations:
\begin{equation}\label{csn5}
\mathcal{A}_-\mid n>=\sum_{n=0}^\infty f^{(p,q)}(n)\mid n><n+1\mid n>=\sum_{n=0}^\infty f^{(p,q)}(n)\mid n>\left[\sum_{s=1}^{2}\delta_{n_s n_s-1}e_s\right]=f^{(p,q)}(n-1)\mid n-1>\nonumber
\end{equation}
and
\begin{equation}\label{csn6}
\mathcal{A}_+\mid n>=\sum_{n=0}^\infty f^{(p,q)}(n)\mid n+1><n\mid n>=f^{(p,q)}(n)\mid n+1>.\nonumber
\end{equation}
Thus the operator $\mathcal{A}_-$ acts as a lowering operator whereas $\mathcal{A}_+$ acts as a raising operator and thus justify their name as annihilation and creation operators. Their product operators in the normal ordered manner are diagonal operators in the basis of Fock states, read
\begin{eqnarray}\label{csn7}
\mathcal{A}_-\mathcal{A}_+\mid n>&=& f^{(p,q)}(n)\mathcal{A}_-\mid n+1>=\left[f^{(p,q)}(n)\right]^2\mid n>\nonumber\\
\Rightarrow <n\mid \mathcal{A}_-\mathcal{A}_+\mid n>&=&\left[f^{(p,q)}(n)\right]^2\nonumber
\end{eqnarray}
and
\begin{eqnarray}\label{csn8}
\mathcal{A}_+\mathcal{A}_-\mid n>&=& f^{(p,q)}(n-1)\mathcal{A}_+\mid n-1>=\left[f^{(p,q)}(n-1)\right]^2\mid n>\nonumber\\
\Rightarrow <n\mid \mathcal{A}_+\mathcal{A}_-\mid n>&=&\left[f^{(p,q)}(n-1)\right]^2.\nonumber
\end{eqnarray}
Thus the ladder operators satisfy a noncommutative commutation relation
\begin{equation}\label{csn9}
\left[\mathcal{A}_-,\mathcal{A}_+\right]=\sum_{n=0}^{\infty}\left(\left[f^{(p,q)}(n)\right]^2-\left[f^{(p,q)}(n-1)\right]^2\right)\mid n><n\mid\nonumber
\end{equation}
in similar to the complex valued coherent states for shifted oscillators \cite{daskal,AS,PP}. On using the recurrence relation (\ref{cs2b}), we can obtain
\begin{eqnarray}\label{BGCS1}
\mathcal{A}_-\mid p;q;Z>&=&\left(\sum_{n=0}^\infty f^{(p,q)}(n)\mid n><n+1\mid\right)\frac{1}{\sqrt{\mathcal{N}^{(p,q)}(\mid Z\mid_h^2)}}\left(\sum_{n=0}^\infty \frac{Z^n}{\sqrt{\rho^{(p,q)}(n)}}\mid n>\right)\nonumber\\
&=&\sum_{n=0}^\infty \frac{f^{(p,q)}(n)}{\sqrt{\mathcal{N}^{(p,q)}(\mid Z\mid_h^2)}}\frac{Z^{n+1}}{\sqrt{\rho^{(p,q)}(n+1)}}\mid n>\nonumber\\
&=&Z\frac{1}{\sqrt{\mathcal{N}^{(p,q)}(\mid Z\mid_h^2)}}\left(\sum_{n=0}^\infty \frac{Z^n}{\sqrt{\rho^{(p,q)}(n)}}\mid n>\right)\nonumber\\
&=&Z\mid p;q;Z>\nonumber
\end{eqnarray}
 which shows that BGHS $\mid p;q;Z>$ are eigenstates of $\mathcal{A}_-$, the annihilation operator with bicomplex eigenvalue $Z$. As a result BGHS satisfy the criterions of being coherent states (or $f$-coherent states) defined in (\ref{cs1}) for bicomplex shifted oscillator problem. It would be highly interesting to investigate physical content like photon statics that includes photon number distribution and mean photon number or phase properties of related applicable phase distributions of these $f$-coherent states in future.
\section{Conclusion}
In this article, the notion of generalized hypergeometric function has been introduced in bicomplex module. The new bicomplex generalized hypergeometric function includes some known functions such as bicomplex hypergeometric function, hypergeometric function and confluent hypergeometric function as particular cases. Some integral representations for this function are deduced and various differential and contiguous relations are also established. Corollary \ref{co:qs} indicates that, by setting suitable particular values of the parameters $p$ and $q$, the integral representation of the new bicomplex generalized hypergeometric function coincides with the integral representation of bicomplex hypergeometric function already obtained in \cite{bc_hypergeometric}. Similarly, for some special cases the results can be reduced to the results involving other special functions. An application of bicomplex generalized hypergeometric function in quantum mechanics is also provided as a consequence.


\begin{thebibliography}{99}
\bibitem{AS}T. Appl, D. H.Schiller, Generalized hypergeometric coherent states, J.Phys.A:Math.Gen., 37 (2004), 2731-2750.
\bibitem{ban}A. Banerjee, On the quantum mechanics of bicomplex Hamiltonian system,  Annals of Physics {\bf 377}  (2017), 493--505.
\bibitem{B2} M. Combescure, D. Robert, \textit{Coherent states and applications in mathematical physics}, Theoretical and Mathematical Physics, Springer (2012).

\bibitem{Das-2019}
S. Das, Inequalities for $q$-gamma function ratios, Anal. Math. Phys., {\bf 9} (2019), no.1, 313–321.

\bibitem{Das-Mehrez-2021}
S. Das, K. Mehrez, On geometric properties of the Mittag-Leffler and Wright functions, J. Korean Math. Soc., {\bf 58} (2021), no.4, 949–965.

\bibitem{Das-Mehrez-2022}
S. Das, K. Mehrez, Geometric properties of the four parameters Wright function, J. Contemp. Math. Anal., {\bf 57} (2022), no.1, 43–58.

\bibitem{daskal}C. Daskaloyannis, Generalized deformed oscillator and nonlinear algebras, J. Phys. A {\bf 24} (1991), L789-L794.
\bibitem{Glauber} R. J. Glauber, Coherent and Incoherent States of the Radiation Field, Physical Review 131 (1963), 2766-2788.
\bibitem{gamma}
S. P. Goyal, T. Mathur\ and\ R. Goyal, Bicomplex gamma and beta functions, J. Rajasthan Acad. Phys. Sci. {\bf 5} (2006), no.~1, 131--142.
\bibitem{polygamma}
R. Goyal, Bicomplex polygamma function, Tokyo J. Math. {\bf 30} (2007), no.~2, 523--530.
\bibitem{bc_root}
W. Johnston\ and\ C. M. Makdad, A comparison of norms: bicomplex root and ratio tests and an extension theorem, Amer. Math. Monthly {\bf 128} (2021), no.~6, 525--533.
\bibitem{laplace transform}
Y.~S. Kim, A.~K. Rathie\ and\ C.~H. Lee, New Laplace transforms for the generalized hypergeometric function $_2F_2$, Honam Math. J. {\bf 37} (2015), no.~2, 245--252.

\bibitem{Lav}R. G. Lavoie, L. Marchildon, D. Rochon, Infinite-dimensional bicomplex Hilbert spaces, Ann. Funct. Anal. 1 (2010), No. 2, 75–91.
\bibitem{bc_number}
M. E. Luna-Elizarrar\'{a}s\ et al., Bicomplex numbers and their elementary functions, Cubo {\bf 14} (2012), no.~2, 61--80.
\bibitem{bc_holomorphic}
M. E. Luna-Elizarrar\'{a}s\ et al., {\it Bicomplex holomorphic functions}, Frontiers in Mathematics, Birkh\"{a}user/Springer, Cham, 2015.
\bibitem{bc_hypergeometric}
R. Meena\ and\ A. K. Bhabor, Bicomplex hypergeometric function and its properties. Integral Transforms Spec. Funct., {\bf 34} (2023), no.~6, 478--494.
 \bibitem{B1} A. M. Perelomov, \textit{Generalized coherent states and their applications}, Texts and monographs in physics, Springer, Berlin (1986).
 \bibitem{PP}D. Popov, M. Popov, Some operational properties of the generalized hypergeometric coherent states, Phys. Scr. 90 (2015),035101.
\bibitem{multicomplex}
G. B. Price, {\it An introduction to multicomplex spaces and functions}, Monographs and Textbooks in Pure and Applied Mathematics, 140, Marcel Dekker, Inc., New York, 1991.
\bibitem{Quesne}C. Quesne, Generalized Coherent States Associated with the $C_\lambda$-Extended Oscillator, Annals of Physics 293 (2001), 147-188.
\bibitem{sp_function}
E. D. Rainville, {\it Special functions}, The Macmillan Company, New York, 1960.
\bibitem{hy_contiguous}
M.~A. Rakha\ and\ A.~K. Ibrahim, On the contiguous relations of hypergeometric series, J. Comput. Appl. Math. {\bf 192} (2006), no.~2, 396--410.
\bibitem{algebraic}
D. Rochon, A bicomplex Riemann zeta function, Tokyo J. Math. {\bf 27} (2004), no.~2, 357--369.
 \bibitem{sch} E. Schr\"{o}dinger, Der stetige Übergang von der Mikro- zur Makromechanik, Naturwissenschaften 14(1926), 664-666,(English translation in Collected Papers on Wave Mechanics. London: Blackie and Sons, 1928,Pages 41-44).
\bibitem{le}
C. Segre, Le rappresentazioni reali delle forme complesse e gli enti iperalgebrici, Math. Ann. {\bf 40} (1892), no.~3, 413--467.

\end{thebibliography}
\end{document}